\documentclass[a4paper,12pt]{article}

\usepackage[a4paper,bindingoffset=0.2in,%
            left=1in,right=1in,top=1in,bottom=1in,%
            footskip=.25in]{geometry}
\usepackage[latin1]{inputenc}
\usepackage[T1]{fontenc}
\usepackage[english]{babel}
\usepackage{amsmath,amssymb,amsthm}
\usepackage{amsfonts}%
\usepackage{amssymb}%
\usepackage{indentfirst}
\usepackage{fancyhdr}
\usepackage{physics}
\usepackage{young}
\usepackage{graphicx}

\usepackage[vcentermath]{youngtab}

\RequirePackage[colorlinks,linkcolor=blue,citecolor=blue,urlcolor=blue]{hyperref}
\usepackage{enumerate}
\usepackage{color}
\usepackage{pst-fill,pst-grad,pst-plot,pst-eucl,pstricks-add,pst-node}
\usepackage{mathdots}
\usepackage{dsfont}

\usepackage{float}
\usepackage{subcaption}
\usepackage{mathrsfs}

\theoremstyle{plain}
\newtheorem{thm}{Theorem}[section]
\newtheorem{lemma}[thm]{Lemma}
\newtheorem{prop}[thm]{Proposition}
\newtheorem{cor}[thm]{Corollary}

\theoremstyle{definition}

\newtheorem{oss}[thm]{Remark}

\newcommand{\R}{\mathbb{R}}
\newcommand{\C}{\mathbb{C}}
\newcommand{\Pp}{\mathbb{P}}

\renewcommand{\dd}{\mathrm{d}}

\newcommand{\E}{\mathbb{E}}

\renewcommand{\Re}{{\rm Re}\,}
\renewcommand{\Im}{{\rm Im}\,}

\allowdisplaybreaks

\usepackage{xspace}

\def\Holk{H^0(M,L^k)\xspace}

\def\N{\mathbb{N}\xspace}

\def\R{\mathbb{R}\xspace}

\newcommand\blfootnote[1]{
  \begingroup
  \renewcommand\thefootnote{}\footnote{#1}
  \addtocounter{footnote}{-1}
  \endgroup
}

\title{Berezin--Toeplitz operators, Kodaira maps, and random sections}
\author{Michele Ancona \thanks{Institut de Recherche Math\'ematique Avanc\'ee, UMR 7501 Universit\'e de Strasbourg et CNRS. 7 rue Ren\'e-Descartes, 67000 Strasbourg, France. {\em E-mail:}  \texttt{michi.ancona@gmail.com}} \and Yohann Le Floch \thanks{Institut de Recherche Math\'ematique Avanc\'ee, UMR 7501 Universit\'e de Strasbourg et CNRS. 7 rue Ren\'e-Descartes, 67000 Strasbourg, France. {\em E-mail:} \texttt{ylefloch@unistra.fr} }}
\date{}

\begin{document}
\maketitle

\begin{abstract}
We study the zeros of sections of the form $T_k s_k$ of a large power $L^{\otimes k} \to M$ of a holomorphic positive Hermitian line bundle over a compact K\"ahler manifold $M$, where $s_k$ is a random holomorphic section of $L^{\otimes k}$ and $T_k$ is a Berezin-Toeplitz operator, in the limit $k \to +\infty$. In particular, we compute the second order approximation of the expectation of the distribution of these zeros. In a ball of radius of order $k^{-\frac{1}{2}}$ around $x \in M$, assuming that the principal symbol $f$ of $T_k$ is real-valued and vanishes transversally, we show that this expectation exhibits two drastically different behaviors depending on whether $f(x) = 0$ or $f(x) \neq 0$. These different regimes are related to a similar phenomenon about the convergence of the normalized Fubini-Study forms associated with $T_k$:  they converge to the K\"ahler form in the sense of currents as $k\rightarrow + \infty$, but not as differential forms (even pointwise). This contrasts with the standard case $f=1$, in which the convergence is in the $\mathscr{C}^{\infty}$--topology. From this, we are able to recover the zero set of $f$ from the zeros of $T_k s_k$.

\end{abstract}

\blfootnote{\,\,\emph{2020 Mathematics Subject Classification.} 81S10, 81Q20, 53D50, 32L05, 32Q15, 33C05, 60G15. \\
\indent\indent \emph{Key words and phrases.} Random sections, Berezin-Toeplitz operators, Currents of integration, Fubini-Study forms.}

\newpage

\section{Introduction}

The motivation for this paper stems from the following inverse problem: given the action of a quantum observable on random quantum states, can one recover properties of the underlying classical observable?

This question is part of the broader goal to study quantum footprints of classical observables, which has been the object of intense research in the last decades. Here we are specifically interested in the following inverse problem: if $f \in \mathscr{C}^{\infty}(M)$ is a classical observable on a phase space $M$ and $T: \mathcal{H} \to \mathcal{H}$ is a quantum observable quantizing $f$ acting on a Hilbert space $\mathcal{H}$, which properties of $f$ can be derived from the study of $T$? This type of inverse problems is often seen from a spectral point of view, as in the seminal article by Kac \cite{kac} dealing with the spectrum of the Laplacian on a planar domain, and the numerous works that it inspired (see for instance the surveys \cite{datchev-hezari-survey13,zelditch-survey}). Here we work in a semiclassical context, which means that the Hilbert spaces and quantum observables depend on a small parameter $\hbar$, and we are interested in the limit $\hbar \to 0$. Inverse spectral problems in this setting have been intensively studied by various authors, see for instance the recent review \cite{VN-footprints} and the references therein. Here we propose another approach, based on the observation of the action of $T$ on quantum states obtained as random combinations of pure states: from the observation of this action for a large number of realizations of the random state, can one infer some properties of $f$?

In this paper we answer this last question positively. More precisely, we are able to recover all the regular levels of $f$ from the zeros of certain random holomorphic sections of (a large power of) a holomorphic line bundle over $M$. 

\subsection{Framework}

We work in the context of geometric quantization \cite{Sou,Kos} and Berezin-Toeplitz operators \cite{Ber,BouGui,BMS,Cha03,ma-marinescu-symplectic}. This means that the phase space $M$ is a compact K\"ahler manifold and that the quantum observables are operators acting on spaces of holomorphic sections $H^0(M,L^{\otimes k})$, where $L \to M$ is a positive line bundle and $k$ is an integer; the semiclassical limit is $k \to +\infty$ (in this setting the small parameter $\hbar$ corresponds to $k^{-1}$). For each $k$, the space $H^0(M,L^{\otimes k})$ is finite-dimensional and carries a natural $L^2$ Hermitian product  $\langle \cdot, \cdot \rangle_{L^2}$ induced by the choice of a positively curved Hermitian metric $h$ on $L$. This Hermitian product is defined as 
\[ \langle\sigma,\tau\rangle_{L^2} = \int_{x\in M}h^k_x(\sigma(x),\tau(x))\frac{\omega^n}{n!} \]
for any $\sigma,\tau\in H^0(M,L^{\otimes k})$, where $\omega = i c_1(L,h)$.
 
\paragraph{Berezin-Toeplitz operators.} To any classical observable $f \in \mathscr{C}^{\infty}(M)$, one can naturally associate a sequence of operators $T_k(f): H^0(M,L^{\otimes k}) \to H^0(M,L^{\otimes k})$ as follows. Let $L^2(M,L^{\otimes k})$ be the Hilbert space obtained as the closure of $\mathscr{C}^{\infty}(M,L^{\otimes k})$ with respect to $\langle\cdot,\cdot\rangle_{L^2}$, and let $\Pi_k: L^2(M,L^{\otimes k}) \to H^0(M,L^{\otimes k})$ be the orthogonal projector from this space to the space of holomorphic sections. Then
\[ T_k(f): s \in H^0(M,L^{\otimes k}) \mapsto \Pi_k(f s) \in H^0(M,L^{\otimes k}).\]
This is an instance of Berezin-Toeplitz operator with principal symbol $f$. More generally, Berezin-Toeplitz operators are operators of the form
\[ T_k = \Pi_k f(\cdot,k) + R_k: H^0(M,L^{\otimes k}) \to H^0(M,L^{\otimes k}) \]
where $(f(\cdot,k))_{k \in \N}$ is a sequence of elements of $\mathscr{C}^{\infty}(M)$ with an asymptotic expansion of the form
\[ f(\cdot,k) = f_0 + k^{-1} f_1 + k^{-2} f_2 + \ldots \]
for the $\mathscr{C}^{\infty}$ topology, and the operator norm of $R_k$ is a $O(k^{-N})$ for every $N \in \N$. The first term $f_0$ in the asymptotic expansion of $f(\cdot,k)$ is called the \emph{principal symbol} of $T_k$.

\paragraph{Random sections and Kodaira maps.} Given a Berezin-Toeplitz operator $T_k$, we study the zeros of $T_k s_k$ where $s_k$ is a random holomorphic section of $L^{\otimes k}$ of the form
\begin{equation} s_k = \sum_{\ell = 1}^{N_k} \alpha_{\ell} e_{\ell}, \quad \alpha_{\ell} \sim \mathcal{N}_{\C}(0,1) \ \text{i.i.d.} \label{eq:rand_sec_intro} \end{equation}
where $N_k = \dim H^0(M,L^{\otimes k})$ and $(e_{\ell})_{1 \leq \ell \leq N_k}$ is any orthonormal basis of $H^0(M,L^{\otimes k})$.
 Such random zeros are related to the properties of some Kodaira maps associated with $T_k$. Before defining those, we recall some facts about the standard Kodaira maps.

 Let $e_1,\dots, e_{N_k}$ be any orthonormal basis of $H^0(M,L^{\otimes k})$. By the Kodaira Embedding Theorem we have that, for $k$ large enough, the base locus $\cap_{s\in H^0(M,L^{\otimes k})}\{s=0\}$ is empty and the Kodaira map $$\Phi_k:x\in M\mapsto [e_1(x):\dots:e_{N_k}(x)]\in\C \mathbb{P}^{N_k-1}$$ is an embedding.
The pull-back $\Phi_k^*\omega_{FS}$  of the Fubini-Study form does not depend on the choice of the orthonormal basis and  its cohomology class $[\Phi_k^*\omega_{FS}]$ equals $k[\omega]$. It is then natural to compare  the forms $\frac{1}{k}\Phi_k^*\omega_{FS}$ and $\omega$ with each other. Tian's asymptotic isometry theorem \cite{bouche,tian,zelditch} says that the former  converges to the latter in the $\mathscr{C}^{\infty}$ topology, as $k\rightarrow + \infty$. More precisely, for any $m\in\mathbb{N}$, we have 
\[ \norm{\frac{1}{k}\Phi_k^*\omega_{FS}-\omega}_{\mathscr{C}^{m}}=O(k^{-1}). \]
Using this result, Shiffman and Zelditch \cite{shiffman-zelditch}  proved that the expected normalized current of integration $Z_s$ of a random section $s\in H^0(M,L^{\otimes k})$ converges weakly in the sense of currents to the K\"ahler form. In this paper, we will give similar results about Kodaira maps and zeros of random sections twisted by Berezin-Toeplitz operators.

\subsection{Main results}\label{sec:main_results}
\subsubsection{(Non)-convergence of the Fubini-Study forms}

Let $e_1,\dots,e_{N_k}$ be any orthonormal basis of $H^0(M,L^{\otimes k})$, and let $T_k$ be a Berezin-Toeplitz operator with principal symbol $f \in \mathscr{C}^{\infty}(M,\R)$. We consider the following ``twisted'' Kodaira map:
\begin{equation} \Phi_{T_k}: M \dashrightarrow \C \mathbb{P}^{N_k-1}, \quad x \mapsto [(T_k e_1)(x): \dots :(T_k e_{N_k})(x)], \label{eq:kodaira}\end{equation}
 which is well-defined outside the locus $\bigcap_{s\in H^0(M,L^{\otimes k})}\{T_k s=0\}$. The first goal of the paper is to give a natural sufficient condition on $f$ for which the map $\Phi_{T_k}$ is everywhere well-defined.

\begin{thm}
\label{embedding}
Assume that the principal symbol $f$ of $T_k$ is a smooth, real-valued function which vanishes transversally. Then, for $k$ large enough, the map $\Phi_{T_k}$ is well-defined on the whole $M$, that is $\bigcap_{s\in H^0(M,L^{\otimes k})}\{T_k s=0\} = \emptyset$. 
\end{thm}

As a  consequence of Theorem \ref{embedding} we have that the pull-back of the Fubini-Study form $\omega_{FS}$ is a \emph{smooth} form defined on the whole $M$ (rather than  just a current). The pull-backed forms $\Phi_{T_k}^*\omega_{FS}$ are usually also called Fubini-Study forms. The cohomology class of $\Phi_{T_k}^*\omega_{FS}$ equals $k[\omega]$; 
it is then natural to ask about the convergence of the sequence of smooth forms $\left( \frac{1}{k}\Phi_{T_k}^*\omega_{FS} \right)_{k \in \N}$. The following theorem deals with the weak convergence of the normalized Fubini-Study forms.

\begin{thm}
\label{convergenceofcurrents}
Assume that the principal symbol $f \in \mathscr{C}^{\infty}(M, \R)$ of $T_k$ is a smooth function vanishing transversally. The sequence of smooth forms $\frac{1}{k}\Phi_{T_k}^*\omega_{FS}$ converges to $\omega$ weakly in the sense of currents. 
\end{thm}

The next result estimates the error term $\frac{1}{k}\Phi_{T_k}^*\omega_{FS} - \omega$, which explicitly involves $f$.

\begin{thm}
\label{thm:error_term}
Let $f \in \mathscr{C}^{\infty}(M, \R)$ be a smooth function vanishing transversally, and let $T_k$ be a Berezin-Toeplitz operator with principal symbol $f$. Then $\log f^2$ is locally integrable and
\[ \Phi_{T_k}^*\omega_{FS} - k \omega \underset{k \to +\infty}{\longrightarrow} i \partial \bar{\partial} \log f^2 \]
in the sense of currents.
\end{thm}

 The following theorem shows that we cannot expect better than the convergence in the sense of currents as soon as $ f^{-1}(0) \neq\emptyset$. This shows a striking difference with the Fubini-Study forms associated with the standard Kodaira maps. As recalled in Section \ref{framework}, the K\"ahler form $\omega$ induces a Riemannian metric on $M$, which by duality induces a metric on $T^*M$, that we denote by $|\cdot|_{\omega}$. 

\begin{thm}
\label{nonconvergence}
Let $f \in \mathscr{C}^{\infty}(M, \R)$ be a smooth function vanishing transversally, and let $T_k$ be a Berezin-Toeplitz operator with principal symbol $f$. Then the sequence $\frac{1}{k}\Phi_{T_k}^*\omega_{FS}$ converges to $\omega$ locally uniformly on  $M\setminus  f^{-1}(0)$ in the $\mathscr{C}^{\infty}$ norm. However, $\frac{1}{k}\Phi_{T_k}^*\omega_{FS}$ does \textit{not} converge to $\omega$ in the $\mathscr{C}^0$-topology (and even pointwise) on $ f^{-1}(0)$. More precisely, 
\[ \left(\frac{1}{k}\Phi_{T_k}^*\omega_{FS}\right)_x - \omega_x \underset{k \to + \infty}{\longrightarrow} \begin{cases} 0 \text{ if } f(x) \neq 0, \\[2mm]  \frac{4 i ( \partial f \wedge \bar{\partial} f )_x}{|\dd f(x)|_{\omega}^2} \text{ if } f(x) = 0. \end{cases} \]
\end{thm}

\subsubsection{Fubini-Study forms at Planck scale}
\label{subsec:planck}

Theorem \ref{nonconvergence} shows that the zero locus of $f$ plays a fundamental role in the non-convergence of the Fubini-Study forms $\frac{1}{k}\Phi_{T_k}^*\omega_{FS}$ to $\omega$ as  differential forms. Indeed the difference $\left(\frac{1}{k}\Phi_{T_k}^*\omega_{FS}\right)_x - \omega_x$ exhibits two very different behaviors on and outside $f^{-1}(0)$. In order to further study these two regimes, a natural idea is to work on a smaller scale which allows us to localize around any given point. We show that the scale $k^{-\frac{1}{2}}$ (that we call Planck scale in the rest of the paper) is well-adapted to this problem. Remark that the Planck scale $k^{-\frac{1}{2}}$ is natural in both quantum mechanics and K\"ahler geometry (it is the scale at which the Bergman kernel displays its universality \cite{bsz,daima,mamabook}).
At this scale, we are able to produce precise asymptotics for the difference $\frac{1}{k}\Phi_{T_k}^*\omega_{FS} - \omega$  in the sense of currents. This is the content of Theorem \ref{thm:error_term_in_k} (for the behavior on $f^{-1}(0)$) and Theorem \ref{thm:error_term_in_k_2} (for the behavior outside $f^{-1}(0)$). Let $n = \dim_{\C} M$.

\begin{thm}
\label{thm:error_term_in_k}
Let $f\in\mathscr{C}^{\infty}(M,\R)$ be a smooth function vanishing transversally, and let $T_k$ be a Berezin-Toeplitz operator with principal symbol $f$. Let $\varphi$ be a smooth $(n-1,n-1)$-form on $M$.  Then, for any $x\in  f^{-1}(0)$ we have
\[ \int_{B(x,\frac{R}{\sqrt{k}})} \left( \Phi_{T_k}^*\omega_{FS} - k \omega \right) \wedge \varphi = k^{-n+1} \frac{2 F_\varphi(x)}{|\dd f(x)|_{\omega}^2} C_n(R)  + O(k^{-n+\frac{1}{2}}).  \]
Here $B(x,\frac{R}{\sqrt{k}})$ is the geodesic ball of radius $\frac{R}{\sqrt{k}}$ around $x$,  $F_{\varphi}$ is the function defined as 
\[ i \partial f\wedge\bar{\partial}f \wedge \varphi = F_{\varphi}\frac{\omega^n}{n!}\]
and $C_n(R)$ is a positive (and explicit) \textit{universal} constant, only depending on $R$ and $n$.
\end{thm} 

The constant $C_n(R)$ in this statement is computed in Proposition \ref{prop:calcul_integrale} and equals
\begin{equation} C_n(R) = \frac{2^n \pi^n (n-1)!}{(2n-2)!} \left(  \sum_{\ell = 0}^{n-1} \binom{n - \frac{3}{2}}{\ell} 2^{\ell} R^{2 \ell} - (1 + 2 R^2)^{n - \frac{3}{2}} \right) \label{eq:Cn(R)} \end{equation}
with $\binom{\alpha}{\ell}  = \frac{\alpha (\alpha - 1) \ldots (\alpha - \ell + 1)}{\ell!}$ for $\alpha \in \R$, $\ell \in \N_{>0}$ and $\binom{\alpha}{0} = 1$. In particular, 
\[ C_1(R) = 2\pi \left( 1 - \frac{1}{\sqrt{1 + 2 R^2}} \right). \]

As can be seen in the course of the proof of Proposition \ref{prop:calcul_integrale}, the constant $C_n(R)$ can also be expressed in terms of hypergeometric functions. This seems to reflect some arithmetic flavor that is a priori surprising (at least for the authors).

Note that when $\varphi =  \frac{\omega^{n-1}}{(n-1)!}$ the first order term in the expansion of Theorem \ref{thm:error_term_in_k} is universal (it depends neither on $f$ nor on $X$, but only on $n$ and $R$). More precisely:

\begin{cor}
\label{cor:univ}
Let $f\in\mathscr{C}^{\infty}(M,\R)$ be a smooth function vanishing transversally, and let $T_k$ be a Berezin-Toeplitz operator with principal symbol $f$. Then, for any $x\in  f^{-1}(0)$ we have
\[ \int_{B(x,\frac{R}{\sqrt{k}})} \left( \Phi_{T_k}^*\omega_{FS} - k \omega \right) \wedge \frac{\omega^{n-1}}{(n-1)!} =  k^{-n+1}  C_n(R)  + O(k^{-n+\frac{1}{2}})  \]
where $C_n(R)$ is as in Equation \eqref{eq:Cn(R)}.
\end{cor}

The next theorem is the analogue of Theorem \ref{thm:error_term_in_k} in the case where the point $x$ is not a zero of $f$.

\begin{thm}
\label{thm:error_term_in_k_2}
Let $f \in \mathscr{C}^{\infty}(M, \R)$ be a smooth function vanishing transversally, and let $T_k$ be a Berezin-Toeplitz operator with principal symbol $f$. Let $\varphi$ be a smooth $(n-1,n-1)$-form on $M$. For any $x\notin f^{-1}(0)$ we have 
\[ \int_{B(x,\frac{R}{\sqrt{k}})} \left( \Phi_{T_k}^*\omega_{FS} - k \omega \right) \wedge \varphi = k^{-n} R^{2n} L_{\varphi}(x) \mathrm{Vol}(B_{\R^{2n}}(0,1)) + O(k^{-n-\frac{1}{2}}).  \]
Here $B(x,\frac{R}{\sqrt{k}})$ is the geodesic ball of radius $\frac{R}{\sqrt{k}}$ centered at $x$, and $L_{\varphi}$ is the function defined as 
\[ i \partial\bar{\partial} \log f^2 \wedge \varphi = L_{\varphi} \frac{\omega^{n}}{n!} .\]
\end{thm} 

It is worth noting the two different behaviors of Theorems \ref{thm:error_term_in_k} and \ref{thm:error_term_in_k_2}. Indeed, if $x \in f^{-1}(0)$, the order of magnitude of $\int_{B(x,\frac{R}{\sqrt{k}})} \left( \Phi_{T_k}^*\omega_{FS} - k \omega \right) \wedge \varphi$ is $O(k^{-n+1})$, whereas if $x \notin f^{-1}(0)$ this order is $O(k^{-n})$. This should be compared to Theorem \ref{nonconvergence}, in which the differential forms $\frac{1}{k}\Phi^*_{T_k}\omega_{FS}$ did not converge exactly on  the zero locus of $f$.
 
\begin{oss}
If $T_k$ is a Berezin-Toeplitz operator with principal symbol $f$, then the operator $T_k - \lambda \mathrm{Id}$ is a Berezin-Toeplitz operator with principal symbol $f - \lambda$. So up to replacing $f$ by $f - \lambda$, we can replace $f^{-1}(0)$ by any regular level of $f$ in the above statements and discussions.
\end{oss} 
 
\subsubsection{Applications to random zeros}
\label{intro:random}

In this section we study how the action of a Berezin-Toeplitz operator affects the zeros of random sections $s\in H^0(M,L^{\otimes k})$. Here, ``random'' is with respect to the natural Gaussian measure $\mu_k$ on $H^0(M,L^{\otimes k})$ given by $\dd\mu_k(s)=\frac{1}{\pi^{N_k}}e^{-\norm{s}^2_{L^2}}\dd s$, where $\dd s$ is the Lebesgue measure on $(H^0(M,L^{\otimes k}),\langle\cdot,\cdot \rangle_{L^2})$ and $N_k = \dim H^0(M,L^{\otimes k})$. Choosing a random element with respect to this probability measure amounts to considering a random linear combination as in Equation \eqref{eq:rand_sec_intro}. Such random holomorphic sections were introduced  in \cite{shiffman-zelditch} and  have been intensively studied since (see for example \cite{bsz,SZvariance,dinhsibony,ComanMarinescu}). This can be seen as a natural geometric generalization of the more classical orthogonal polynomials. 

Given a holomorphic section $s\in H^0(M,L^{\otimes k})$, we denote by $Z_s$ the current of integration on $\{s=0\}$. This is defined by its action on smooth $(n-1,n-1)-$forms $\varphi$ as $\langle Z_s,\varphi\rangle=\int_{\{s=0\}}\varphi$.

Given a Berezin-Toeplitz operator $T_k$, we will be interested in the current-valued random variable $s\in H^0(M,L^{\otimes k})\mapsto Z_{T_ks}$.
Recall that the expected value $\E[ Z_{T_ks}]$ of $Z_{T_ks}$ is defined by the formula
\[ \E[\langle Z_{T_ks},\varphi\rangle]=\int_{s\in H^0(M,L^{\otimes k})}\left(\int_{\{T_ks=0\}}\varphi\right) \dd \mu_k(s) \]
for any smooth $(n-1,n-1)-$form $\varphi$. 
For the basic case $f=1$ (which corresponds to  $T_k=\text{Id}$), such an expected current of integration has been studied in \cite{shiffman-zelditch}, where it is shown that $\frac{1}{k}\E\big[\langle Z_s,\varphi\rangle\big]\rightarrow \frac{1}{2\pi} \int_{M}\omega\wedge\varphi$. Note that the factor $2\pi$ does not appear in \cite{shiffman-zelditch} due to a different convention (the volume of a complex projective line equals $1$ in \cite{shiffman-zelditch} and $2\pi$  in the present paper). 
The following result is a generalization of  \cite{shiffman-zelditch} when the random section is perturbated by a Berezin-Toeplitz operator.
\begin{thm}
\label{thm expectedvalue} Let $f \in \mathscr{C}^{\infty}(M,\R)$ be a smooth function vanishing transversally and let $T_k$ be a Berezin-Toeplitz operator with principal symbol $f$. Then  
\[ \frac{1}{k}\E[Z_{T_k s}] \underset{k \to +\infty}{\longrightarrow} \frac{\omega}{2\pi} \]
 weakly in the sense of currents. Moreover, we have 
\[ \E[Z_{T_k s}] - \frac{k \omega}{2\pi} \underset{k \to +\infty}{\longrightarrow} \frac{i}{2\pi} \partial \bar{\partial} \log f^2 \]
weakly in the sense of currents.
\end{thm}

As in Section \ref{subsec:planck}, if we look at the  Planck scale $k^{-\frac{1}{2}}$ we can obtain much more precise asymptotics.

\begin{thm}
\label{thm random error} 
Let $f\in\mathscr{C}^{\infty}(M,\R)$ be a smooth function vanishing transversally, and let $T_k$ be a Berezin-Toeplitz operator with principal symbol $f$. Let $x \in M$. Let $\varphi$ be a smooth $(n-1,n-1)$-form on $M$. For every $R > 0$, 
\[ \int_{B(x,\frac{R}{\sqrt{k}})} \left( \E[Z_{T_k s}] - \frac{k}{2\pi} \omega \right) \wedge \varphi = \begin{cases} k^{-n+1}  \frac{ F_\varphi(x)}{\pi  |\dd f(x)|_{\omega}^2} C_n(R) + O(k^{-n+\frac{1}{2}})  \text{ if } x \in f^{-1}(0), \\[2mm]
 k^{-n} \frac{R^{2n} L_{\varphi}(x) \mathrm{Vol}(B_{\R^{2n}}(0,1))}{2\pi}  + O(k^{-n-\frac{1}{2}}) \text{ if } x \notin f^{-1}(0). \end{cases}  \]
Here $B(x,\frac{R}{\sqrt{k}})$, $F_{\varphi}$, $L_{\varphi}$ and $C_n(R)$ are as in Theorems \ref{thm:error_term_in_k} and \ref{thm:error_term_in_k_2}. 
\end{thm}

Theorem \ref{thm expectedvalue} and Theorem \ref{thm random error} follow from Theorems \ref{convergenceofcurrents}, \ref{thm:error_term}, \ref{thm:error_term_in_k} and \ref{thm:error_term_in_k_2} after the remark that $\E[Z_{T_k s}]$ and  $\frac{1}{2\pi}\Phi_{T_k}^*\omega_{FS}$ are equal  as  currents, see Lemma \ref{expectedforanyk}.

As for Theorems \ref{thm:error_term_in_k} and \ref{thm:error_term_in_k_2}, it is worth noting the two different behaviors of $\int_{B(x,\frac{R}{\sqrt{k}})} \left( \E[Z_{T_k s}] - \frac{k}{2\pi} \omega \right) \wedge \varphi$ for the cases $f(x)=0$ and $f(x)\neq 0$. Indeed, in the first case, this is of order $O(k^{-n+1})$ whereas if $f(x)\neq 0$ this is of order $O(k^{-n})$. This suggests that the locus of zeros of $T_ks_k$ tends to concentrate a little more on $ f^{-1}(0)$. This is confirmed by numerical simulations that we will show in Section \ref{sec:numerics}.

\begin{oss}
Note that if we replace $T_k$ with principal symbol $f$ by $S_k = \lambda T_k$ for some $\lambda \in \R \setminus \{0\}$, the principal symbol of $S_k$ is $g = \lambda f$. So the zero sets of $f$ and $g$ coincide, and if $s$ is a holomorphic section of $L^{\otimes k}$, the zeros of $T_k s$ and $S_k s$ agree. So the quantities that appear in all our results should be invariant by this scaling; one readily checks that they indeed are.
\end{oss}

\subsubsection*{Organization of the paper} The paper is organized as follows. In Section \ref{background} we recall the context and some useful properties of Berezin-Toeplitz operators; this also serves to introduce our notation and conventions. In Section \ref{Sec:proofs} we prove Theorems \ref{embedding}, \ref{convergenceofcurrents}, \ref{thm:error_term} and \ref{nonconvergence}. In Section \ref{sec:planck} we prove Theorems \ref{thm:error_term_in_k} and \ref{thm:error_term_in_k_2}. In Section \ref{Sec:Random} we explain why the previous theorems imply Theorems \ref{thm expectedvalue} and \ref{thm random error}, and check the validity of our results by performing some numerical simulations. In Appendix \ref{sect:appendix} we give a proof of Theorem \ref{symbol_comp}.

\subsubsection*{Acknowledgments} The research leading to this manuscript was mainly performed while the first author was a postdoctoral fellow of the Labex IRMIA. Both authors thank the Labex and the Institut de Recherche Math\'ematique Avanc\'ee for the opportunity to work in these excellent conditions.

\section{Background}
\label{background}

In this section we recall some necessary background in K\"ahler geometry and Berezin-Toeplitz operators. For more details about the latter, see for instance \cite{Schli,LFbook} and the references therein. The main result that we state in this section is the positivity of the Schwartz kernel of $T_k^* T_k $ on the diagonal when the principal symbol $f$ of $T_k$ vanishes transversally; this will be a key ingredient in the proof of our main results. 

\subsection{Framework}
\label{framework}

Let $(M,\omega)$ be a $n$-dimensional compact  K\"ahler manifold such that $[\frac{\omega}{2\pi}]\in H^2(M,\mathbb{Z})$ and let $(L,h)$ be a Hermitian line bundle whose Chern curvature $c_1(L,h)$ equals $-i\omega$. We recall that this curvature is locally defined by $-\partial\bar{\partial}\log h(e_L,e_L)$, where $e_L$ is any local non-vanishing holomorphic section of $L$. 

\paragraph{Induced metrics.} Let $j$ be the complex structure on $TM$ and let $G = \omega(\cdot, j \cdot)$ be the Riemannian metric induced by $\omega$ and $j$ on $TM$; by extending it by sesquilinearity, we obtain an Hermitian metric on $TM \otimes \C$, which in turn induces an Hermitian metric on $T^*M \otimes \C$ by duality. We still denote by $G$ these metrics, and, when the context is clear, we use $|\cdot|_\omega$ for the pointwise norm associated with $G$. If $\alpha \in (T^{1,0}M)^*$ and $\beta \in (T^{0,1}M)^*$, then $G(\alpha, \beta) = 0$. In local holomorphic coordinates $z_1 = x_1 + i y_1$, $\ldots$, $z_n = x_n + i y_n$, we have $\omega = \frac{i}{2} \sum_{\ell,m = 1}^n G_{\ell,m} dz_{\ell} \wedge \dd \bar{z}_m$ and 
\begin{equation} \forall \ell, m \in \{1, \ldots, n\} \quad G\left(\dd z_{\ell}, \dd z_m \right) = 2 G^{\ell,m}, \quad G(\dd x_{\ell}, \dd x_m) = G(\dd y_{\ell}, \dd y_m) = G^{\ell,m}  \label{eq:norme} \end{equation}
where 
\[ \forall \ell, m \in \{1, \ldots, n\} \qquad \sum_{p=1}^n G_{\ell,p} G^{m,p} = \delta_{\ell,m}. \]
Using this expression, one readily checks that
\[ \forall \alpha, \beta \in T^*M \otimes \C \qquad G(\bar{\alpha}, \bar{\beta}) = \overline{G(\alpha, \beta)}. \]
To avoid confusion, we denote by $| \cdot |_{\omega}$ the norm induced by this metric. Moreover, we consider the holomorphic Laplacian on $M$, which reads in these local coordinates
\[ \Delta = 2 \sum_{\ell,m = 1}^n G^{\ell,m} \partial_{z_{\ell}} \partial_{\bar{z}_m}. \]

\paragraph{$L^2$-Hermitian products.} 

For any positive $k\in \mathbb{N}$, we denote by $h^k$ the Hermitian metric on $L^k := L^{\otimes k}$ induced by the metric $h$ on $L$. Remark that for this induced metric we have $c_1(L^k,h^k)=kc_1(L,h)$. 

For any positive $k\in \mathbb{N}$, the space of global holomorphic sections $H^0(M,L^k)$ is naturally equipped with the $L^2$-Hermitian product $\langle\cdot,\cdot \rangle_{L^2}$  defined by 
\[ \langle\sigma,\tau\rangle_{L^2} = \int_{x\in M}h^k_x(\sigma(x),\tau(x))\frac{\omega^n}{n!} \]
for any $\sigma,\tau\in H^0(M,L^k)$. It is standard that the space $H^0(M,L^k)$ is finite-dimensional, with dimension 
\[ N_k = \left( \frac{k}{2\pi} \right)^n \text{Vol}(M) + O(k^{n-1}) \]
where $n = \dim_{\C} M$ and $\text{Vol}(M)$ is the volume of $M$ computed with respect to the volume form $\dd\text{Vol} = \frac{\omega^n}{n!}$ induced by $\omega$.

\subsection{Berezin-Toeplitz operators}
\label{subsec:bto}

For $k \in \N$, let $L^2(M,L^k)$ be the Hilbert space obtained as the closure of $\mathscr{C}^{\infty}(M,L^k)$ with respect to $\langle\cdot,\cdot\rangle_{L^2}$, and let $\Pi_k: L^2(M,L^k) \to H^0(M,L^k)$ be the orthogonal projector from this space to the space of holomorphic sections. For any smooth function $f \in \mathscr{C}^{\infty}(M)$, the Berezin-Toeplitz operator associated with $f$ is the endomorphism $T_k(f) = \Pi_k f : H^0(M,L^k) \to H^0(M,L^k)$. More generally, Berezin-Toeplitz operators are operators of the form
\[ T_k = \Pi_k f(\cdot,k) + R_k: \Holk \to \Holk \]
where $(f(\cdot,k))_{k \in \N}$ is a sequence of elements of $\mathscr{C}^{\infty}(M)$ with an asymptotic expansion of the form
\[ f(\cdot,k) = f_0 + k^{-1} f_1 + k^{-2} f_2 + \ldots \]
for the $\mathscr{C}^{\infty}$ topology, and the operator norm of $R_k$ is a $O(k^{-N})$ for every $N \in \N$. The first term $f_0$ in the asymptotic expansion of $f(\cdot,k)$ is called the \emph{principal symbol} of $T_k$. Similarly, we will call $f_1$ the \emph{subprincipal symbol} of $T_k$ (it is the contravariant subprincipal symbol, see for instance \cite{Cha03}). Recall that any Berezin-Toeplitz operator  $T_k$ has a Schwartz kernel, which is a holomorphic section of $L^k \boxtimes \bar{L}^k \to M \times \overline{M}$ that we still denote by $T_k$. This means that for any $s \in H^0(M,L^k)$ and for any $x \in M$,
\[ (T_k s)(x) = \int_M T_k(x,y) s(y) \ \dd\text{Vol}(y). \]
Recall that if $e_1, \dots, e_{N_k}$ is any orthonormal basis of $H^0(M,L^k)$, then
\begin{equation} \forall x,y \in M \qquad T_k(x,y) = \sum_{\ell = 1}^{N_k} (T_k e_{\ell})(x) \otimes \overline{e_{\ell}(y)}. \label{eq:def_kernel}\end{equation}

In order to prove our main results described in Section \ref{sec:main_results}, we will need to compute the subprincipal term in the asymptotic expansion of the Schwartz kernel of the product of two Berezin-Toeplitz operators. This is the content of Theorem  \ref{symbol_comp}. This is nowadays a standard result and can be found for example in \cite[Formula $(0.16)$]{ma-marinescu-toeplitz} or derived from \cite{Cha03}. Since our notation differs from both these references, for the sake of completeness, we will give a proof of this result following \cite{Cha03} in Appendix \ref{sect:appendix} .

\begin{thm}
\label{symbol_comp} 
Let $T_k, S_k$ be Berezin-Toeplitz operators with respective real-valued principal symbols $f_0, g_0 \in \mathscr{C}^{\infty}(M,\R)$ and subprincipal symbols $f_1, g_1 \in \mathscr{C}^{\infty}(M)$. Then the on-diagonal expansion of the Schwartz kernel of the Berezin-Toeplitz operator $B_k = T_k S_k$ reads
\[ B_k(x,x) = \left( \frac{k}{2\pi} \right)^n \left( b_0(x) + k^{-1} b_1(x) + O(k^{-2}) \right) \]
where the $O(k^{-2})$ is uniform on $M$, with $b_0 = f_0 g_0$ and
\[ b_1 = f_0 g_1 + f_1 g_0 + f_0 \Delta g_0 + g_0 \Delta f_0 + \frac{r}{2} f_0 g_0 + G(\partial g_0, \partial f_0).  \]
Here, $r$ denotes the scalar curvature of $M$ and $G$ is the metric on $T^* M$ defined in Section \ref{framework}. 
\end{thm}

\begin{lemma}
\label{lm:norme_df}
Let $f \in \mathscr{C}^{\infty}(M)$. Then $|\dd f|_{\omega}^2 = 2 |\partial f|_{\omega}^2$. 
\end{lemma}

\begin{proof}
Since $\dd f = \partial f + \bar{\partial} f$, we have that
\[ |\dd f|_{\omega}^2 = |\partial f|_{\omega}^2 + 2 \Re G(\partial f, \bar{\partial} f) + |\bar{\partial} f|_{\omega}^2 =  |\partial f|_{\omega}^2 + |\bar{\partial} f|_{\omega}^2 \]
since $\partial f \in \Omega^{(1,0)}(M)$ and $\bar{\partial} f \in \Omega^{(0,1)}(M)$. Moreover, 
\[ |\bar{\partial} f|_{\omega}^2 = G(\bar{\partial} f, \bar{\partial} f) = \overline{G(\partial f, \partial f)} = G(\partial f, \partial f) = |\partial f|_{\omega}^2. \]
\end{proof}

\begin{cor}\label{secondterm} Let $T_k$ be a Berezin-Toeplitz operator with real-valued principal symbol $f \in \mathscr{C}^{\infty}(M,\R)$ and subprincipal symbol $g \in \mathscr{C}^{\infty}(M)$. Then the on-diagonal expansion of the Schwartz kernel of the Berezin-Toeplitz operator $B_k = T_k^* T_k$ equals 
\[ \left( \frac{k}{2\pi} \right)^n \left( f^2 + k^{-1} b_1 + O(k^{-2}) \right)  \]
where the remainder $O(k^{-2})$ is uniform on $M$ and
\[ b_1 = 2 f \Re(g) + 2 f \Delta f +  \frac{r}{2} f^2 + \frac{1}{2} |\dd f|_{\omega}^2.   \]
\end{cor}

\begin{proof}
The principal symbol of $T_k^*$ is $\bar{f}$, and its subprincipal symbol is $\bar{g}$. So Theorem \ref{symbol_comp} yields
\[ b_1 = 2 f \Re(g) + 2 f \Delta f +  \frac{r}{2} f^2 + |\partial f|_{\omega}^2 \] 
and Lemma \ref{lm:norme_df} gives the result.
\end{proof}

The previous result implies the following crucial fact that will be key in the proof of all our main results: if $f$ vanishes transversally, the kernel of $B_k$ on the diagonal is always strictly positive. More precisely:

\begin{cor}
\label{positivity} Let $T_k$ be a Berezin-Toeplitz operator with real-valued principal symbol $f$ and let $B_k = T_k^* T_k$. If $f$ vanishes transversally, there exists $c>0$ such that, for $k$ large enough and for any $x\in M$, we have $B_k(x,x) > c k^{n-1}$.
\end{cor}

\begin{proof}
By Corollary \ref{secondterm} we have the following uniform asymptotics:
\[ B_k(x,x)= \left( \frac{k}{2\pi} \right)^n (b_0(x) + k^{-1} b_1(x) + O(k^{-2})) \]
where $b_0(x)=|f(x)|^2$ and 
\[ b_1 = 2 f \Re(g) + 2 f \Delta f +  \frac{r}{2} f^2 + \frac{1}{2} |\dd f|_{\omega}^2  \]
with $g$ the subprincipal symbol of $T_k$. Since $f$ vanishes transversally, we have that
\[ b_1(x) =  \frac{1}{2} |\dd f(x)|_{\omega}^2 > 0 \]
for $x \in f^{-1}(0)$, so $b_1$ is strictly positive in a neighborhood $U$ of $ f^{-1}(0)$. Let $c$ be the minimum of $b_1$ on $U$ and $C$ be the minimum of $|f|^2$ on $M\setminus U$. We then have  $B_k(x,x)\geq C k^n + O(k^{n-1})$ outside $U$ and $B_k(x,x)\geq c k^{n-1}+O(k^{n-2})$ on $U$. Hence the result. 
\end{proof}
The next corollary is equivalent to the previous one. However, we prefer to put a separate statement because we will use it repeatedly throughout the paper.

\begin{cor}
\label{positivity2} 
Let $T_k$ be a Berezin-Toeplitz operator with real-valued principal symbol $f$ and let $B_k = T_k^* T_k$. If $f$ vanishes transversally, then there exists $c>0$ such that, for any $k$ large enough  we have $\abs{f}^2+k^{-1}b_1 > c k^{-1}$ and $\abs{f}^2 + \frac{k^{-1}}{2} |\mathrm{d}f|_{\omega}^{2} > c k^{-1}$.
\end{cor}

\section{Kodaira maps and Fubini-Study forms}
\label{Sec:proofs}

The goal of this section is to study the Kodaira map associated with $T_k$. First we prove that it is well-defined for $k$ large enough; this is Theorem \ref{embedding}. Then we prove Theorems \ref{convergenceofcurrents} and \ref{thm:error_term} which deal with the convergence in the sense of currents of the associated Fubini-Study form. Finally, we show the non-convergence of this form in the sense of differential forms, that is Theorem \ref{nonconvergence}. We follow the notation introduced in Section \ref{background}.

\subsection{The Kodaira map is well-defined }

In this section we prove that the Kodaira map $\Phi_{T_k}$ defined in Equation \eqref{eq:kodaira} is well-defined everywhere on $M$ for $k$ large enough (see Theorem \ref{embedding}). This follows from the combination of the the positivity result for the Schwartz kernel of $T_k^* T_k$ (see Corollary \ref{positivity}) and the next lemma. 

\begin{lemma}
\label{sumbasis2}
Let $A\in \textrm{End}\big(H^0(M,L^k)\big)$ and $e_1, \dots, e_{N_k}$ be any orthonormal basis of $H^0(M,L^k)$. Then, for any $x,y\in M$ we have the equality 
\[ \sum_{\ell=1}^{N_k}(A e_{\ell})(x)\otimes \overline{(A e_{\ell})(y)} = \sum_{\ell=1}^{N_k}(A^* A e_{\ell})(x) \otimes \overline{e_{\ell}(y)}.\]
\end{lemma}

\begin{proof}
For any $\ell, p \in \{1,N_k\}$, let $A_{\ell,p} = \langle A e_{\ell},e_p \rangle$. Then 
\[ \begin{split}  \sum_{\ell=1}^{N_k}(A e_{\ell})(x)\otimes \overline{(A e_{\ell})(y)} & =  \sum_{\ell=1}^{N_k} \left(  \sum_{p=1}^{N_k} A_{\ell,p} e_p(x) \right) \otimes \left(  \sum_{q=1}^{N_k} \overline{A_{\ell,q}} \ \overline{e_q(y)} \right) \\
& =  \sum_{\ell=1}^{N_k}  \sum_{p=1}^{N_k}  \sum_{q=1}^{N_k} A_{\ell,p} \overline{A_{\ell,q}} e_p(x) \otimes \overline{e_q(y)} \\
& =  \sum_{\ell=1}^{N_k}  \sum_{p=1}^{N_k}  \sum_{q=1}^{N_k} A_{\ell,p} A^*_{q,\ell} e_p(x) \otimes \overline{e_q(y)}  \\
& =  \sum_{p=1}^{N_k}  \sum_{q=1}^{N_k} (A^*A)_{q,p} e_p(x) \otimes \overline{e_q(y)} \\
& = \sum_{q=1}^{N_k} (A^*A e_q)(x) \otimes \overline{e_q(y)}.  \end{split} \]
\end{proof}

\begin{proof}[Proof of Theorem \ref{embedding}]
The map $\Phi_{T_k}$ is well-defined everywhere on $M$ if and only if $\bigcap_{i=1}^{N_k}\{T_k s_i=0\}=\emptyset$. In order to prove this we will show that there exists an integer $k_0$ such that, for any $x\in M$ and any $k \geq k_0$, the quantity $\sum_{\ell = 1}^{N_k} \abs{(T_k e_{\ell})(x)}^2_k$ is strictly positive. By Equation \eqref{eq:def_kernel} and Lemma \ref{sumbasis2}, $\sum_{\ell = 1}^{N_k}\abs{(T_k e_{\ell})(x)}^2_k$ equals the value on the diagonal of the Schwartz kernel $B_k$ of $T_k^* T_k$, that is $\sum_{\ell = 1}^{N_k}\abs{(T_k e_{\ell})(x)}^2_k = B_k(x,x)$. By Corollary \ref{positivity}, for $k$ large enough, $B_k(x,x)$ is strictly positive, hence the result.
\end{proof}

\subsection{(Non)-convergence of the Fubini-Study forms}
In this section, we prove Theorem \ref{convergenceofcurrents}, which deals with the convergence of the normalized Fubini-Study forms $\frac{1}{k} \Phi^*_{T_k}\omega_{FS}$ in the sense of currents, and Theorem \ref{nonconvergence}, which instead says that such forms do not converge in the sense of differential forms. 
 Throughout this section, $T_k$ is a Berezin-Toeplitz operator with real-valued principal symbol $f$ and subprincipal symbol $g$. As above, the on-diagonal expansion of the Schwartz kernel of $B_k = T_k^* T_k$ is denoted by
\[ B_k(x,x) =   \left( \frac{k}{2\pi} \right)^n (b_0(x) + k^{-1} b_1(x) + O(k^{-2})) \]
where $b_0$ and $b_1$ are given by Corollary \ref{secondterm}. In what follows we will use the slightly abusive notation $B_k$ for the restriction of $B_k$ to the diagonal; we will never need to evaluate this kernel away from the diagonal.

\begin{proof}[Proof of Theorem \ref{convergenceofcurrents}]
The following equality of smooth forms is standard:
\begin{equation}\label{equalofforms}
i\partial\bar{\partial}\log B_k = \Phi^*_{T_k}\omega_{FS}-k\omega
\end{equation}
so that, in order to prove the theorem,  we have  to show that $\frac{1}{k}\partial\bar{\partial}\log B_k$ goes to $0$ in the sense of currents as $k \to +\infty$. 

Remark that we have the equality $\partial\bar{\partial}\log B_k = \partial\bar{\partial}\log\big((2 \max \abs{f}^2 \left(\frac{k}{2\pi}\right)^{n})^{-1} B_k\big)$. Moreover, for $k$ large enough, we have 
\[ ck^{-1} < \left(2 \max \abs{f}^2 \left(\frac{k}{2\pi}\right)^{n}\right)^{-1} B_k < 1,\]
where the left-hand inequality follows from Corollary \ref{positivity} and the right-hand one from Corollary \ref{secondterm}.

For any  $(n-1,n-1)$ smooth form  $\varphi$ on $M$, let us denote by $f_{\varphi}$ the function given by the equality $\partial\bar{\partial}\varphi=f_{\varphi}\frac{\omega^n}{n!}$ and by $\norm{\partial\bar{\partial}\varphi}_{\infty}$ the sup-norm of $f_{\varphi}$. We then have

\[ \begin{split} \abs{\int_{M} \partial\bar{\partial}\log B_k\wedge \varphi} & = \abs{\int_{M} \partial\bar{\partial}\log \big((2 \max \abs{f}^2k^{n})^{-1} B_k\big)\wedge \varphi} \\
& \leq \int_{M}\abs{\partial\bar{\partial}\log \big((2 \max \abs{f}^2k^{n})^{-1} B_k\big)\wedge \varphi} \\
& \leq \int_{M} \abs{\log (ck^{-1})\partial\bar{\partial}\varphi} \\
& = \int_{M} \abs{\log (ck^{-1})f_{\varphi}\frac{\omega^n}{n!}} \\
& = \norm{\partial\bar{\partial}\varphi}_{\infty}O(\log k).\end{split} \]

This implies that, for any $(n-1,n-1)$ smooth form  $\varphi$, the quantity $\int_{M}\frac{1}{k}\log B_k \partial\bar{\partial}\varphi$ goes to $0$ as $k\rightarrow + \infty$, which exactly means that $\frac{1}{k}\partial\bar{\partial}\log B_k$ goes to $0$ in the sense of currents as $k\rightarrow + \infty$. Hence the result.
\end{proof}

\begin{proof}[Proof of Theorem \ref{nonconvergence}]
We start by proving that the sequence $\frac{1}{k}\Phi_{T_k}^*\omega_{FS}$ converges to $\omega$ locally uniformly on  $M\setminus  f^{-1}(0)$ in the $\mathscr{C}^{\infty}$ norm.
Remark that the equality \eqref{equalofforms} 
implies that this is equivalent to showing that for any relatively compact open set $U\subset M\setminus  f^{-1}(0)$ and for any $m\in\mathbb{N}$, we have $\norm{\partial\bar{\partial}\log B_k}_{\mathscr{C}^m,U}=O(1)$. Now, we know that $B_k(x) = \left(\frac{k}{2\pi}\right)^n(f^2+O(k^{-1}))$ uniformly, so that we have 
\[ \partial\bar{\partial}\log B_k =\partial\bar{\partial}\log (f^2+O(k^{-1})) \]
uniformly.  The result follows from the fact that  there exists two positive constants $c_U$ and $C_U$ such that $c_U\leq f^2 \leq C_U$ on $U$, so that 
\[ \partial\bar{\partial}\log (f^2+O(k^{-1})) = \partial\bar{\partial}\log ( f^2(1+O(k^{-1})))=\partial\bar{\partial}\log (f^2)+O(k^{-1})\]
on $U$. This shows that $\norm{\partial\bar{\partial}\log B_k}_{\mathscr{C}^m,U}=O(1)$, which was our goal.

Let us now prove that the smooth form $\frac{1}{k}\partial\bar{\partial}\log B_k$ does not tend to $0$ on $ f^{-1}(0)$ as $k\rightarrow +\infty$.
We have 
\begin{equation} \begin{split} \partial\bar{\partial}\log B_k & = \partial\bar{\partial}\log(f^2 + k^{-1} b_1+O(k^{-2})) \\
& = \partial\bigg(\frac{2f\bar{\partial}f + k^{-1} \bar{\partial} b_1+O(k^{-2})}{f^2 + k^{-1} b_1}\bigg) \\ 
& = \frac{(2\partial f\wedge\bar{\partial} f+2f\partial\bar{\partial}f + k^{-1}\partial\bar{\partial}b_1)}{f^2 + k^{-1} b_1} \\
& - \frac{(2f\bar{\partial}f + k^{-1}\bar{\partial}b_1)\wedge(2f\partial f+\partial k^{-1} b_1)  + O(k^{-2})}{(f^2 + k^{-1} b_1)^2}.  \end{split} \label{eq:logBk} \end{equation}
If we evaluate this form at a point $x$ where $f$ vanishes we obtain
\[ \begin{split} \left(\partial\bar{\partial}\log B_k\right)_x & = \frac{(2 (\partial f\wedge\bar{\partial} f)_x + k^{-1} (\partial\bar{\partial}b_1)_x) k^{-1} b_1(x) - k^{-2} ( \bar{\partial}b_1 \wedge \partial b_1)_x  + O(k^{-2})}{b_1(x)^2k^{-2}} \\
& = \frac{2k (\partial f\wedge\bar{\partial} f)_x}{b_1(x)} + O(1) \\
& =  \frac{4k (\partial f\wedge\bar{\partial} f)_x}{|\dd f(x)|_{\omega}^2} + O(1).\end{split} \]
At a point where $f$ vanishes we then have that
$\frac{1}{k}\partial\bar{\partial}\log B_k$ tends to $\frac{4\partial f\wedge\bar{\partial} f}{|\dd f|_{\omega}^2}$
which is non zero as $f$ vanishes transversally. Moreover, at a point where $f$ does not vanish, Equation \eqref{eq:logBk} shows that $\partial\bar{\partial}\log B_k = O(1)$, so $\frac{1}{k}\partial\bar{\partial}\log B_k$ goes to zero as $k \to +\infty$.
\end{proof}

\subsection{Convergence of the error term}
In this section, we prove Theorem \ref{thm:error_term}, which estimates (in the sense of currents) the error term $\frac{1}{k} \Phi^*_{T_k}\omega_{FS}-\omega$.
We start with a lemma, which is actually part of the statement of Theorem \ref{thm:error_term}.
\begin{lemma}\label{integrability of logf} Let $f:M\rightarrow\R$ be a smooth function vanishing transversally. Then $\log f^2$ is (locally) integrable, so $\partial\bar{\partial}\log f^2$ is a well-defined current.
\end{lemma}
\begin{proof}

As $M$ is compact, it is enough to show that $\log f^2$ is locally integrable. Locally around a point  $x$ where $f(x)\neq 0$ there is nothing to prove since $\log f^2$ is locally bounded there.

Let us consider a point $x\in  f^{-1}(0)$. By Hadamard's lemma, we can find a small neighborhood $U$ of $x$ and local (real) coordinates $x_1,\dots,x_{2n}$, in which $f^{-1}(0)$ becomes $\{x_1 = 0\}$, such that $f(x_1,\dots,x_{2n})=x_1g(x_1,\dots,x_{2n}) + h(x_1,\dots,x_{2n})$, where the function $g$ is smooth with $g(0) \neq 0$ and $h(x_1,\dots,x_{2n}) = O(x_1^2 + \ldots + x_{2n}^2)$. Up to replacing $U$ with a smaller neighborhood, we can assume that $U$ is of the form $(-\epsilon,\epsilon)^{2n}$.
We then have $\log f^2 = \log x_1^2 + \log g^2 + O(1)$, so that
\[ \int_{U}\log f^2=\int_{(-\epsilon,\epsilon)^{2n}}\log x_1^2 \ \mathrm{d}x_1\cdots\mathrm{d}x_{2n} + \int_{(-\epsilon,\epsilon)^{2n}}\log g^2 \ \mathrm{d}x_1\cdots\mathrm{d}x_{2n} +O(1). \]
The integral $\int_{(-\epsilon,\epsilon)^{2n}}\log g^2 \ \mathrm{d}x_1\cdots\mathrm{d}x_{2n}$ is bounded as $g^2$ is bounded from below by a positive constant. 
The integral $\displaystyle\int_{(-\epsilon,\epsilon)^{2n}}\log x_1^2 \ \mathrm{d}x_1\cdots\mathrm{d}x_{2n}$ equals $(2\epsilon)^{2n-1}\displaystyle\int_{(-\epsilon,\epsilon)}\log x_1^{2}\mathrm{d}x_1$, which is also finite as $\log x_1^{2}$ is integrable around $0$. Hence the result.
\end{proof}
\begin{proof}[Proof of Theorem \ref{thm:error_term}]
Recall that, as currents,
\[ \Phi_{T_k}^*\omega_{FS} - k \omega = i \partial \bar{\partial} \log B_k. \]
For any smooth $(n-1,n-1)$-form $\varphi$ we then get
\[\int_{M}\left(\Phi_{T_k}^*\omega_{FS} - k \omega\right)\wedge\varphi=i\int_M\log B_k \ \partial \bar{\partial}\varphi\]
so that we have to prove the following convergence
\begin{equation}\label{what to prove}
\int_M\log B_k \ \partial \bar{\partial}\varphi \underset{k \to +\infty}{\longrightarrow} \int_{M}\log f^2\ \partial \bar{\partial}\varphi.
\end{equation}

Recall that by Corollary \ref{secondterm}, we have $\log B_k=\log\left(\frac{k}{2\pi}\right)^n+\log(f^2+k^{-1}b_1+O(k^{-2}))$, so that 
\begin{equation}\label{firstequality in proof}
\int_{M}\log B_k\partial \bar{\partial}\varphi=\int_{M}\log(f^2+k^{-1}b_1+O(k^{-2}))\partial \bar{\partial}\varphi.
\end{equation}
By Corollary \ref{positivity2}, we get 
\[\log(f^2+k^{-1}b_1+O(k^{-2}))=\log\big((f^2+k^{-1}b_1)(1+O(k^{-1}))\big)=\log(f^2+k^{-1}b_1)+O(k^{-1})\] 
so that 
\begin{equation}\label{secondequality in proof}
\int_{M}\log(f^2+k^{-1}b_1+O(k^{-2}))\partial \bar{\partial}\varphi=\int_{M}\log(f^2+k^{-1}b_1)\partial \bar{\partial}\varphi+O(k^{-1}).
\end{equation}
In order to prove \eqref{what to prove}, we then have to show that 
\begin{equation}\label{what to prove 2}
\int_M\log(f^2+k^{-1}b_1)\partial \bar{\partial}\varphi  \underset{k \to +\infty}{\longrightarrow} \int_{M} \log f^2\partial \bar{\partial}\varphi.
\end{equation}
 For this, we will partition $M$ into two subsets. For this, remark that since $b_1 = \frac{1}{2} \abs{\mathrm{d}f}_{\omega}^2 >0$ on $\Sigma:=  f^{-1}(0)$, we can find a positive $\epsilon$ such that $b_1$ is strictly positive on an $\epsilon$-tubular neighborhood $\Sigma_{\epsilon}$ of $\Sigma$.
We can then write 
\begin{equation}\label{decomposition}
\int_{M}\log(f^2+k^{-1}b_1)\partial \bar{\partial}\varphi=\int_{M\setminus \Sigma_{\epsilon}}\log(f^2+k^{-1}b_1)\partial \bar{\partial}\varphi+\int_{ \Sigma_{\epsilon}}\log(f^2+k^{-1}b_1)\partial \bar{\partial}\varphi
\end{equation}
For the first integral in the right-hand side of \eqref{decomposition}, remark that $\log (f^2+k^{-1}b_1)$ converges to $\log f^2$ uniformly on $M\setminus \Sigma_{\epsilon}$ and then 
\begin{equation}\label{first term}
\int_{M\setminus \Sigma_{\epsilon}}\log(f^2+k^{-1}b_1)\partial \bar{\partial}\varphi\xrightarrow[k \to +\infty]{} \int_{M\setminus \Sigma_{\epsilon}}\log(f^2)\partial \bar{\partial}\varphi.
\end{equation}
It remains to prove that 
\begin{equation}\label{what to prove 3}
\int_{\Sigma_{\epsilon}}\log(f^2+k^{-1}b_1)\partial \bar{\partial}\varphi\xrightarrow[k \to +\infty]{}  \int_{ \Sigma_{\epsilon}}\log f^2\partial \bar{\partial}\varphi.
\end{equation}

By the choice of $\epsilon$, the function $f^2+k^{-1}b_1$ is strictly positive on $\Sigma_{\epsilon}$. Moreover, up to taking a smaller $\epsilon$, we can suppose that $f^2+k^{-1}b_1<1$ on $\Sigma_{\epsilon}$, for $k$ large enough.
Let us write $\partial \bar{\partial}\varphi=\psi\omega^n$, for $\psi$ a smooth function on $M$.
We have the pointwise convergence $\log(f^2+k^{-1}b_1)\psi\rightarrow\log(f^2)\psi$. Moreover, for $k$ large enough, we have $$\abs{\log(f^2+k^{-1}b_1)\psi}\leq \abs{\log (f^2)}\sup\abs{\psi}.$$
By Lemma \ref{integrability of logf}, the function $\log (f^2)$ is integrable, so by Lebesgue's dominated convergence theorem, we obtain the convergence \eqref{what to prove 3}. Hence the result.
\end{proof}

\section{Estimates at Planck scale}
\label{sec:planck}

This section is organized as follows. In Section \ref{subsec:onthezero} we prove Theorem \ref{thm:error_term_in_k} and  Corollary \ref{cor:univ} and in Section \ref{subsec:outsidethezero} we prove Theorem \ref{thm:error_term_in_k_2}. 
We will need the following notation and lemma in both Sections \ref{subsec:onthezero} and \ref{subsec:outsidethezero}. For any $(n-1,n-1)$-form $\varphi$, any $R>0$, any $k\in \mathbb{N}$ and any point $x\in M$, we denote by $\varphi_{x,R,k}$ the  $(n-1,n-1)$-form $\chi_{B(x,\frac{R}{\sqrt{k}})}\varphi$, where $\chi_{B(x,\frac{R}{\sqrt{k}})}$ is the characteristic function of the geodesic ball $B(x,\frac{R}{\sqrt{k}})$. For a smooth $(1,1)$-form $\psi$, we denote by $\langle \psi,\varphi_{x,R,k}\rangle$ the natural pairing $\int_{B(x,\frac{R}{\sqrt{k}})}\psi\wedge\varphi$.

\begin{lemma}\label{lemme sur tout M} Let $\varphi$ be a smooth $(n-1,n-1)$-form and let $R>0$.
Then, for any $x\in M$, we have
 $$\langle\partial\bar{\partial}\log B_k,\varphi_{x,R,k}\rangle = \langle\partial\bar{\partial}\log (f^2+k^{-1}b_1),\varphi_{x,R,k}\rangle +O(k^{-n-1})$$
 as $k \rightarrow +\infty$.
\end{lemma}
\begin{proof} Recall that $B_k = \big(\frac{k}{2\pi}\big)^n(f^2+k^{-1}b_1+O(k^{-2}))$ (see Corollary \ref{secondterm}), so that 
$$\partial\bar{\partial}\log B_k = \partial\bar{\partial}\log \left( \left(\frac{k}{2\pi}\right)^n(f^2+k^{-1}b_1+O(k^{-2})) \right) =\partial\bar{\partial}\log \left(f^2+k^{-1}b_1+O(k^{-2})\right).$$
Now, by Corollary \ref{positivity2}, we can write $f^2+k^{-1}b_1+O(k^{-2})=(f^2+k^{-1}b_1)(1+O(k^{-1}))$, so that we obtain 
\[
\langle\partial\bar{\partial}\log (f^2+k^{-1}b_1+O(k^{-2})),\varphi_{x,R,k}\rangle=\langle\partial\bar{\partial}\log \big((f^2+k^{-1}b_1)(1+O(k^{-1})\big),\varphi_{x,R,k}\rangle.
\]
 The latter equals 
\[\langle\partial\bar{\partial}\log (f^2+k^{-1}b_1),\varphi_{x,R,k}\rangle + \langle\partial\bar{\partial}\log (1+O(k^{-1})),\varphi_{x,R,k}\rangle.\]
We obtain the result by remarking that $\partial\bar{\partial}\log (1+O(k^{-1}))=O(k^{-1})$ and that $\varphi_{x,R,k}$ satisfies  $\mathrm{Vol}(\mathrm{Supp}\big(\varphi_{x,R,k})\big)=O(k^{-n})$.
\end{proof}

\subsection{Estimates on the zero set}\label{subsec:onthezero}

In this section, we prove Theorem \ref{thm:error_term_in_k} and Corollary \ref{cor:univ}. We also compute the universal constant $C_n(R)$ appearing in these results; this is done in Proposition \ref{prop:calcul_integrale}
\begin{lemma}\label{lemma sur les zeros} Let $\varphi$ be a smooth $(n-1,n-1)$-form and let $R>0$.
Then, for any $x\in  f^{-1}(0)$, we have
\[ \langle\partial\bar{\partial}\log (f^2+k^{-1}b_1),\varphi_{x,R,k}\rangle = 4 \int_{B(x,\frac{R}{\sqrt{k}})}\frac{k^{-1} |\mathrm{d}f|^{2}_{\omega} - 2 f^2}{(2 f^2+ k^{-1} |\mathrm{d}f|_{\omega}^{2})^2}\partial f\wedge\bar{\partial}f\wedge\varphi+O(k^{-n+\frac{1}{2}}) \]
as $k \rightarrow +\infty$. 
\end{lemma}

\begin{proof} We start by developing $\partial\bar{\partial}\log (f^2+k^{-1}b_1)$ and obtain
\[ \partial\bar{\partial}\log (f^2+k^{-1}b_1) = \frac{(f^2+k^{-1}b_1)\partial\bar{\partial}(f^2+k^{-1}b_1) - \partial (f^2+k^{-1}b_1)\wedge \bar{\partial}(f^2+k^{-1}b_1)}{(f^2+k^{-1}b_1)^2}. \]
We now use that
\[ f^2 \partial\bar{\partial} f^2 - \partial f^2 \wedge \bar{\partial} f^2 = - 2 f^2 \partial f \wedge \bar{\partial} f + 2 f^3 \partial \bar{\partial} f, \]
that $\abs{f}=O(k^{-1/2})$ and that $b_1 = \frac{1}{2} |\mathrm{d}f |_{\omega}^{2}+O(k^{-1})$ on $B(x,\frac{R}{\sqrt{k}})$ (see Corollary \ref{secondterm}) to obtain, after expanding and collecting the lower order terms, that on $B(x,\frac{R}{\sqrt{k}})$ 
\begin{equation}\label{reste in lemma 2}
\partial\bar{\partial}\log (f^2+k^{-1}b_1) = \frac{(k^{-1}\abs{\mathrm{d}f}_{\omega}^{2} - 2 f^2)}{\big(f^2+ \frac{k^{-1}}{2}(\abs{\mathrm{d}f}_{\omega}^{2} + O(k^{-1}))\big)^{2}}\partial f\wedge\bar{\partial}f + O(k^\frac{1}{2}).
\end{equation}
By Corollary \ref{positivity2}, we have that
\[ f^2 + \frac{k^{-1}}{2} (  \abs{\mathrm{d}f}_{\omega}^{2} + O(k^{-1}))=(f^2 + \frac{k^{-1}}{2} \abs{\mathrm{d}f}_{\omega}^{2})(1 + O(k^{-1})),\]
hence Equation \eqref{reste in lemma 2} yields
\[ \partial\bar{\partial}\log (f^2+k^{-1}b_1) = \frac{ 4 (k^{-1}\abs{\mathrm{d}f}_{\omega}^{2} -  2 f^2)}{\big(2 f^2 + k^{-1}\abs{\mathrm{d}f}_{\omega}^{2}\big)^{2}}\partial f\wedge\bar{\partial}f + O(k^\frac{1}{2}) \]
on $B(x,\frac{R}{\sqrt{k}})$, whose volume is a $O(k^{-n})$, hence the result.
\end{proof}

\begin{lemma}
\label{lemma sur les zeros 2} 
Let $\varphi$ be a smooth $(n-1,n-1)$-form and let $R>0$.
Then, for any $x\in  f^{-1}(0)$, we have
\[  i \int_{B(x,\frac{R}{\sqrt{k}})} \frac{k^{-1}\abs{\mathrm{d}f}_{\omega}^{2} -  2 f^2}{(2 f^2 + k^{-1} \abs{\mathrm{d}f}_{\omega}^{2})^2} \partial f\wedge\bar{\partial}f \wedge \varphi = \frac{k^{-n+1}  F_{\varphi(x)}}{ |\dd f(x)|_{\omega}^2} \int_{B_{\R^{2n}}(0,R)}\frac{1 - 2 t_1^2}{(1 + 2 t_1^2)^2} \dd \lambda(t) +O(k^{-n+\frac{1}{2}}) \]
as $k \rightarrow +\infty$, where $F_{\varphi}$ is such that $i \partial f\wedge\bar{\partial}f \wedge \varphi = F_{\varphi}\frac{\omega^n}{n!}$ and $\dd\lambda = \mathrm{d}t_1\dots \mathrm{d}t_{2n}$.
\end{lemma}

\begin{proof}
Let $x \in f^{-1}(0)$ and let
\begin{equation} \label{eq:integralIk}
 I_k = i \int_{B(x,\frac{R}{\sqrt{k}})} \frac{k^{-1}\abs{\mathrm{d}f}_{\omega}^{2} - 2 f^2}{(2 f^2 + k^{-1}\abs{\mathrm{d}f}_{\omega}^{2})^2} \partial f\wedge\bar{\partial}f \wedge \varphi 
 \end{equation}
be the integral that we want to estimate. Let $z_1 = x_1 + i y_1, \ldots, z_n = x_n + i y_n$ be normal holomorphic coordinates at $x$, defined on some open set $U$ (and, thus, on the ball $B(x,\frac{R}{\sqrt{k}})$  for any $k$ large enough). 
  Recall that these normal holomorphic coordinates at $x$ have the property that 
\begin{equation} \omega = \frac{i}{2} \sum_{\ell,m=1}^n G_{\ell,m} \dd z_{\ell} \wedge \dd\bar{z}_m  \label{eq:normal} \end{equation}
with $(G_{\ell,m})_{1 \leq \ell,m \leq n} =  \mathrm{Id} + O(|z|^2)$. 

Since a unitary linear map sends normal holomorphic coordinates to normal holomorphic coordinates, we may assume that $\frac{\partial f}{\partial x_1}(0) \neq 0$ and that the kernel of $\dd f(x)$ is $\text{Span}(\partial_{y_1}, \partial_{x_2}, \ldots, \partial_{y_n} )$. By Hadamard's lemma, this implies that there exists smooth functions $g_1, \ldots, g_{2n}$ such that 
\[ f(x_1, y_1, \ldots, x_n, y_n) = \sum_{\ell=1}^n ( x_{\ell} g_{\ell}(x_1, y_1, \ldots, x_n, y_n) +  y_{\ell} g_{n + \ell}(x_1, y_1, \ldots, x_n, y_n)  ) \]
for every $z=(x_1, y_1, \ldots, x_n, y_n)$ and $g_2(0) = \ldots = g_{2n}(0) = 0$. Hence, on $B(x,\frac{R}{\sqrt{k}})$ we have
\begin{equation} f(x_1, y_1, \ldots, x_n, y_n) = x_1 g_1(x_1, y_1, \ldots, x_n, y_n) + O(|z|^2). \label{eq:hadamard} \end{equation}
On the ball $B(x,\frac{R}{\sqrt{k}})$, we also have the estimate
\begin{equation} |\dd f(z)|_\omega^2 =  \abs{\frac{\partial f}{\partial x_1}(z)}^{2} + O(|z|^2)  \label{eq:abs_df} \end{equation}
because of the definition of $|\cdot|_\omega$ (see Equation \eqref{eq:norme}) and Equation \eqref{eq:normal}. Using Equation \eqref{eq:abs_df} we then obtain that
\[  \frac{k^{-1}\abs{\mathrm{d}f}_\omega^{2} - 2 f^2}{(2 f^2 + k^{-1}\abs{\mathrm{d}f}_\omega^{2})^2} =   \frac{ k^{-1}  \abs{\frac{\partial f}{\partial x_1}}^{2} - 2 f^2}{\left(2 f^2 + k^{-1} \abs{\frac{\partial f}{\partial x_1}}^{2}\right)^2} + O(k |z|^2),  \]
where, in the denominator, we have used that $\left(2 f^2 + k^{-1} \abs{\frac{\partial f}{\partial x_1}}^{2}\right)^2 \geq c k^{-2}$, see Corollary \ref{positivity2}. So we obtain that the integral \eqref{eq:integralIk} that we want to estimate is equal to
\begin{equation}\label{eq:integralIk2}
 I_k  = i \int_{B_k} \frac{k^{-1}  \abs{\frac{\partial f}{\partial x_1}}^{2} - 2 f^2}{\left(2 f^2 + k^{-1} \abs{\frac{\partial f}{\partial x_1}}^{2}\right)^2} \partial f\wedge\bar{\partial}f \wedge \varphi  + i \int_{B_k} O(k |z|^2) \partial f\wedge\bar{\partial}f \wedge \varphi 
 \end{equation}
where $B_k$ denotes the ball $B(x,\frac{R}{\sqrt{k}})$ in the coordinates $z_1, \ldots, z_n$. Since in these coordinates, the Riemannian metric is the standard metric up to $O(|z|^2)$, there exists $c > 0$ such that $B_{\R^{2n}}(0,\frac{R}{\sqrt{k}} (1  - \frac{c}{\sqrt{k}}) ) \subset B_k \subset B_{\R^{2n}}(0,\frac{R}{\sqrt{k}} (1  + \frac{c}{\sqrt{k}}))$.  The second integral on the right-hand side of Equation \eqref{eq:integralIk2} can then be estimated as 
\[ i \int_{B_k} O(k |z|^2) \partial f\wedge\bar{\partial}f \wedge \varphi = O(1) \text{vol}\left( B_{\R^{2n}}(0,\frac{R}{\sqrt{k}} ) \right) = O(k^{-n}) \]
and moreover for the first integral on the right-hand side of Equation  \eqref{eq:integralIk2} we have
\begin{equation} J_k(-c) \leq i \int_{B_k} \frac{k^{-1}  \abs{\frac{\partial f}{\partial x_1}}^{2} - 2 f^2}{\left(2 f^2 + k^{-1} \abs{\frac{\partial f}{\partial x_1}}^{2}\right)^2} \partial f\wedge\bar{\partial}f \wedge \varphi  \leq J_k(c) \label{eq:gendarmes} \end{equation}
with 
\begin{equation}\label{eq:integralJk}
J_k(\pm c) = i \int_{B_{\R^{2n}}(0,\frac{R}{\sqrt{k}} (1  \pm\frac{c}{\sqrt{k}})  )} \frac{k^{-1}  \abs{\frac{\partial f}{\partial x_1}}^{2} - 2 f^2}{\left(2 f^2 + k^{-1} \abs{\frac{\partial f}{\partial x_1}}^{2}\right)^2} \partial f\wedge\bar{\partial}f \wedge \varphi.  
\end{equation}
Now on $U$, thanks to Equation \eqref{eq:normal}, we have the estimate
\[ i \partial f\wedge\bar{\partial}f \wedge \varphi = F_{\varphi} \frac{\omega^n}{n!} =  F_{\varphi}  \dd x_1 \wedge \dd y_1 \ldots \dd x_n \wedge \dd y_n  \left( 1 + O(|z|^2)  \right),  \]
 which, put into Equation \eqref{eq:integralJk}, gives us
\[ J_k(c)  =  \int_{B_{\R^{2n}}(0,\frac{R}{\sqrt{k}} (1  + \frac{c}{\sqrt{k}})  )} \frac{k^{-1} \abs{\frac{\partial f}{\partial x_1}(z) }^{2} - 2 f(z)^2}{\left( 2 f(z)^2 + k^{-1} \abs{\frac{\partial f}{\partial x_1}(z)}^{2}\right)^2} F_{\varphi}(z) \left( 1 + O(|z|^2) \right) \dd \lambda(z)   \]
with $\dd\lambda = \dd x_1 \dd y_1 \ldots \dd x_n \dd y_n$. The change of variables $w = z \sqrt{k}$ yields
\[ J_k(c) =  k^{-n} \int_{D_k} \frac{k^{-1} \abs{\frac{\partial f}{\partial x_1}(k^{-\frac{1}{2}} w) }^{2} - 2 f(k^{-\frac{1}{2}} w)^2}{\left(2 f(k^{-\frac{1}{2}} w)^2 + k^{-1} \abs{\frac{\partial f}{\partial x_1}(k^{-\frac{1}{2}} w)}^{2}\right)^2} F_{\varphi}(k^{-\frac{1}{2}} w) \left( 1 + O(k^{-1} |w|^2) \right) \dd \lambda(w), \]
where $D_k$ denotes the Euclidian ball   $B_{\R^{2n}}(0, R (1  + \frac{c}{\sqrt{k}}) )$.

Since by Equation \eqref{eq:hadamard},
\[ f(k^{-\frac{1}{2}} w)^2 = k^{-1}  t_1^2 g_1(k^{-\frac{1}{2}} w)^2 + O(k^{-\frac{3}{2}} |w|^2)  \]
this gives
\begin{equation}\label{eq:integralJk2}
 J_k(c) =  k^{-n+1} \int_{D_k} \frac{\abs{\frac{\partial f}{\partial x_1}(k^{-\frac{1}{2}} w) }^{2} - 2 t_1^2 g_1(k^{-\frac{1}{2}} w)^2}{\left(2 t_1^2 g_1(k^{-\frac{1}{2}} w)^2 +  \abs{\frac{\partial f}{\partial x_1}(k^{-\frac{1}{2}} w)}^{2}\right)^2} F_{\varphi}(k^{-\frac{1}{2}} w) \left( 1 + O(k^{-\frac{1}{2}} |w|^2) \right) \dd \lambda(w).  
 \end{equation}
We now treat the function inside the integral appearing in Equation \eqref{eq:integralJk2}.
By Taylor's formula, we have
\[  \frac{\abs{\frac{\partial f}{\partial x_1}(k^{-\frac{1}{2}} w) }^{2} - 2 t_1^2 g_1(k^{-\frac{1}{2}} w)^2}{\left(2 t_1^2 g_1(k^{-\frac{1}{2}} w)^2 +  \abs{\frac{\partial f}{\partial x_1}(k^{-\frac{1}{2}} w)}^{2}\right)^2} F_{\varphi}(k^{-\frac{1}{2}} w) =  \frac{\abs{\frac{\partial f}{\partial x_1}(0)}^{2} - 2 t_1^2 g_1(0)^2}{\left( 2 t_1^2 g_1(0)^2 + \abs{\frac{\partial f}{\partial x_1}(0)}^{2}\right)^2} F_{\varphi}(0) + O(k^{-\frac{1}{2}} |w|). \]
Putting the latter in \eqref{eq:integralJk2}, using the fact that $\mathrm{Vol}(D_k \setminus B_{\R^{2n}}(0,R)) = O(k^{-\frac{1}{2}})$, the equality $\frac{\partial f}{\partial x_1}(0) = g_1(0)$ and Equation \eqref{eq:abs_df},
we finally obtain that
\[ J_k(c) = k^{-n+1} \frac{ F_{\varphi}(0)}{ |\dd f(0)|_{\omega}^2} \int_{B_{\R^{2n}}(0,R)}  \frac{1 - 2 t_1^2}{(1 + 2 t_1^2)^2}\mathrm{d}t_1\dots \mathrm{d}t_{2n} +O(k^{-n+\frac{1}{2}}).  \]
Since the same holds for $J_k(-c)$ (by the same reasoning), we conclude thanks to Equation \eqref{eq:gendarmes}.
\end{proof} 

To conclude the proof of Theorem \ref{thm:error_term_in_k}, we need to compute explicitly the integral in the above lemma.

\begin{prop}
\label{prop:calcul_integrale}
For every $R \geq 0$,
\[  \int_{B_{\R^{2n}}(0,R)}\frac{1 - 2 t_1^2}{(1 + 2 t_1^2)^2}\mathrm{d}t_1\dots \mathrm{d}t_{2n} = \frac{2^{n-1} \pi^n (n-1)!}{(2n-2)!} \left( P_n(2 R^2) - (1 + 2 R^2)^{n - \frac{3}{2}} \right)  \]
where $P_n$ is the Taylor polynomial of order $n-1$ of $g_n: x \mapsto (1 + x)^{n - \frac{3}{2}}$ at $x=0$. More explicitly, 
\[ P_n(X) = \sum_{\ell = 0}^{n-1} \binom{n - \frac{3}{2}}{\ell} X^{\ell}, \quad \binom{\alpha}{\ell}  = \frac{\alpha (\alpha - 1) \ldots (\alpha - \ell + 1)}{\ell!} \text{ for } \ell \in \N_{> 0}, \quad \binom{\alpha}{0} = 1. \]
Moreover, for every $R > 0$, 
\[  \int_{B_{\R^{2n}}(0,R)}\frac{1 - 2 t_1^2}{(1 + 2 t_1^2)^2}\mathrm{d}t_1\dots \mathrm{d}t_{2n} > 0. \]
\end{prop}

\begin{proof}
The change of variables $u = t  \sqrt{2}$ yields that the integral that we want to compute equals $2^{-n} I_n(R \sqrt{2})$ with
\[ I_n(R) = \int_{B_{\R^{2n}}(0,R)}\frac{1 - u_1^2}{(1 + u_1^2)^2}\mathrm{d}u_1\dots \mathrm{d}u_{2n} \]
Using spherical coordinates (in principle we need to assume that $n \geq 2$ for the rest of the proof, but for the case $n=1$ everything works in a similar way), we write
\[ I_n(R) = \int_{D} \frac{1 - r^2 \cos^2 \theta_1}{(1 + r^2 \cos^2 \theta_1)^2} r^{2n-1} \sin^{2n-2} \theta_1 \sin^{2n-3} \theta_2 \ldots \sin \theta_{2n-2} \ \dd r \ \dd\theta_1 \ldots \dd\theta_{2n-1}.   \]
where $D = [0,R] \times ]0,\pi[^{2n-2} \times ]0,2\pi[$. Hence we obtain that 
\begin{equation} I_n(R) = 2^{2n-1} J_n(R) \prod_{\ell = 0}^{2n-3} W_{\ell} = 2^{2n-1} J_n(R) \frac{\pi^{n-1} (n-1)!}{(2n-2)!} \label{eq:I_const} \end{equation}
with $W_{\ell}$ the $\ell$-th Wallis integral and 
\[ J_n(R) = \int_0^R \int_0^{\pi} \frac{1 - r^2 \cos^2 \theta}{(1 + r^2 \cos^2 \theta)^2} r^{2n-1} \sin^{2n-2} \theta \ dr \ \dd\theta. \]
So we are left with computing $J_n$. The change of variables $r = R u$ yields
\[ J_n(R) = R^{2n} \int_0^1 \int_0^{\pi} \frac{1 - R^2 u^2 \cos^2 \theta}{(1 + R^2 u^2 \cos^2 \theta)^2} r^{2n-1} \sin^{2n-2} \theta \ \dd u \ \dd\theta.  \]
But, if $D$ is the quarter of the unit disc in $\R^2$ contained in the upper-right quadrant, then 
\[  \int_0^1 \int_0^{\pi} \frac{1 - R^2 u^2 \cos^2 \theta}{(1 + R^2 u^2 \cos^2 \theta)^2} u^{2n-2} \sin^{2n-2} \theta \ u \ \dd u \ \dd\theta = 2 \int_{D} \frac{1 - R^2 t_1^2}{(1 + R^2 t_1^2)^2} t_2^{2n-2} \ \dd t_1 \ \dd t_2  \]
so we obtain by writing $D = \{ (t_1,t_2) \ | \ 0 \leq t_1 \leq 1, 0 \leq t_2 \leq \sqrt{1 - t_1^2} \}$ and by integrating with respect to $t_2$ that
\[  J_n(R) = \frac{2 R^{2n}}{2n-1} \int_0^1 \frac{1 - R^2 t_1^2}{(1 + R^2 t_1^2)^2} (1 - t_1^2)^{\frac{2n-1}{2}} \ \dd t_1.  \]
The change of variables $u = t_1^2$ yields
\[ \begin{split}  J_n(R) & = \frac{R^{2n}}{2n-1} \int_0^1 \frac{(1 - R^2 u) u^{-\frac{1}{2}} (1 - u)^{\frac{2n-1}{2}}}{(1 + R^2 u)^2}  \ \dd u \\
& = \frac{R^{2n}}{2n-1} \left(  \int_0^1 \frac{u^{-\frac{1}{2}} (1 - u)^{\frac{2n-1}{2}}}{(1 + R^2 u)^2}  \ \dd u - R^2 \int_0^1 \frac{u^{\frac{1}{2}} (1 - u)^{\frac{2n-1}{2}}}{(1 + R^2 u)^2}  \ \dd u  \right). \end{split} \]
Using Equations $(15.6.1)$ and $(15.1.2)$ in \cite{NIST}, we can rewrite these integrals in terms of the hypergeometric function $F = {_2F_1}$:
\[ J_n(R)  = \frac{R^{2n} \Gamma(n+ \frac{1}{2})}{2n-1} \left( \frac{\Gamma(\frac{1}{2}) }{n!} F\left(2, \frac{1}{2}; n+1; -R^2\right)  -  \frac{\Gamma(\frac{3}{2}) R^2 }{(n+1)!} F\left(2, \frac{3}{2}; n+2; -R^2\right) \right) \]
and with the standard explicit expressions for $\Gamma$ at half-integers we obtain that 
\begin{equation} J_n(R) = \frac{R^{2n} (2n-2)! \pi }{ 2^{2n - 1} n! (n-1)!} K_n(R) \label{eq:J_hypergeo} \end{equation}
where
\begin{equation} K_n(R) =  F\left(2, \frac{1}{2}; n+1; -R^2\right)  -  \frac{R^2}{2(n+1)} F\left(2, \frac{3}{2}; n+2; -R^2\right). \label{eq:K_hypergeo} \end{equation}
Now, Equation $(15.5.15)$ in \cite{NIST} with $a= 1$, $b = \frac{3}{2}$, $c = n + 2$ and $z = -R^2$ yields
\[ F\left(2, \frac{3}{2}; n+2; -R^2\right) = (n+1) F\left(1, \frac{3}{2}; n+1; -R^2\right) - n F\left(1, \frac{3}{2}; n+2; -R^2\right)  \]
and Equation $(15.5.16)$ in \cite{NIST} (together with the fact that $F$ is symmetric with respect to $a$ and $b$) with $a = 1$, $b = \frac{3}{2}$, $c = n + 1$ and $z = -R^2$ gives
\[ n R^2 F\left(1, \frac{3}{2}; n+2; -R^2\right) = (n+1) \left( (1 + R^2) F\left(1, \frac{3}{2}; n+1; -R^2\right) -  F\left(1, \frac{1}{2}; n + 1; -R^2\right) \right).  \]
After injecting these two equations in \eqref{eq:K_hypergeo} and simplifying, we get
\[ K_n(R) = F\left(2, \frac{1}{2}; n+1; -R^2\right)  -  \frac{1}{2} F\left(1, \frac{1}{2}; n + 1; -R^2\right) + \frac{1}{2} F\left(1, \frac{3}{2}; n + 1; -R^2\right)  \]
which reduces, using Equation $(15.5.12)$ in \cite{NIST} with $a = 1$, $b = \frac{1}{2}$, $c = n + 1$ and $z = -R^2$, to
\[ K_n(R) = F\left(1, \frac{3}{2}; n + 1; -R^2\right).  \]
Substituting this in Equation \eqref{eq:J_hypergeo}, we finally obtain that 
\begin{equation} J_n(R)  = \frac{R^{2n} (2n-2)! \pi }{ 2^{2n - 1} n! (n-1)!} F\left(1, \frac{3}{2}; n + 1; -R^2\right). \label{eq:J_final} \end{equation}

We claim that this gives the desired result. In order to prove this, we use the integral expression for the remainder in Taylor's formula to write
\[ P_n(R^2) - (1 + R^2)^{n - \frac{3}{2}} = - \frac{R^{2n}}{(n-1)!} \int_0^1 (1 - t)^{n-1} g_n^{(n)}(R^2 t) \ \dd t. \]
But one readily checks that 
\[ \forall x \geq 0 \qquad g_n^{(n)}(x) = -\frac{(2n-2)!}{2^{2n-1} (n-1)!} (1 + x)^{-\frac{3}{2}}, \]
so the above equation becomes
\begin{equation} P_n(R^2) - (1 + R^2)^{n - \frac{3}{2}} = \frac{R^{2n} (2n-2)!}{2^{2n-1} (n-1)!^2} \int_0^1 \frac{(1 - t)^{n-1}}{(1 + R^2 t)^{\frac{3}{2}}} \ \dd t. \label{eq:reste_taylor} \end{equation}
Using once again Equations $(15.6.1)$ and $(15.1.2)$ in \cite{NIST}, we see after simplification that 
\[ P_n(R^2) - (1 + R^2)^{n - \frac{3}{2}} = \frac{R^{2n} (2n-2)!}{2^{2n-1} n! (n-1)!} F\left(1, \frac{3}{2}; n + 1; -R^2\right).  \]
Comparing this with Equation \eqref{eq:J_final} gives
\[ J_n(R) = \pi \left( P_n(R^2) - (1 + R^2)^{n - \frac{3}{2}} \right)  \]
and we conclude using Equation \eqref{eq:I_const} that
\[ I_n(R) =  \frac{2^{2n-1} \pi^{n} (n-1)!}{(2n-2)!}  \left( P_n(R^2) - (1 + R^2)^{n - \frac{3}{2}} \right), \]
as desired.

The fact that $\int_{B_{\R^{2n}}(0,R)}\frac{1 - 2 t_1^2}{(1 + 2 t_1^2)^2}\mathrm{d}t_1\dots \mathrm{d}t_{2n}$ is positive for every $R > 0$ is clear from the expression given in Equation \eqref{eq:reste_taylor}.
\end{proof}

We are now able to prove Theorem \ref{thm:error_term_in_k} and Corollary \ref{cor:univ}.

\begin{proof}[Proof of Theorem \ref{thm:error_term_in_k}]
Recall that, as currents,
\[ \Phi_{T_k}^*\omega_{FS} - k \omega = i \partial \bar{\partial} \log B_k, \]
so we have
\begin{equation}\label{in proof of error term 1}
\langle\left( \Phi_{T_k}^*\omega_{FS} - k \omega \right),\varphi_{x,R,k}\rangle = i\langle\partial \bar{\partial} \log B_k,\varphi_{x,R,k}\rangle.
\end{equation}
By Lemma \ref{lemme sur tout M} and by \eqref{in proof of error term 1} we have that
\begin{equation}\label{in proof of error term 2}
\langle\left( \Phi_{T_k}^*\omega_{FS} - k \omega \right),\varphi_{x,R,k}\rangle=i\langle\partial\bar{\partial}\log (f^2+k^{-1}b_1),\varphi_{x,R,k}\rangle +O(k^{-n-1}).
\end{equation}
The result then follows from Equation \eqref{in proof of error term 2}, Lemma \ref{lemma sur les zeros} and Lemma \ref{lemma sur les zeros 2}.
\end{proof}

\begin{proof}[Proof of Corollary \ref{cor:univ}]
In view of the statement, we need to prove that when $\varphi = \frac{\omega^{n-1}}{(n-1)!}$, the function $F_{\varphi}$ defined as $i \partial f \wedge \bar{\partial} f \wedge \varphi = F_{\varphi} \frac{\omega^n}{n!}$ satisfies
\[ F_{\varphi} = \frac{1}{2} |\dd f|_{\omega}^2. \]
As this is a pointwise property, we may assume that $(M,\omega) = (\C^n, \omega_{\C^n})$ with $\omega_{\C^n} = \frac{i}{2} \sum_{\ell = 1}^n \dd z_{\ell} \wedge \dd \bar{z}_{\ell}$. Then one readily checks that 
\[ i \partial f \wedge \bar{\partial} f \wedge \frac{\omega^{n-1}}{(n-1)!} = \frac{i^n}{2^{n-1}} \sum_{\ell = 1}^n \left( \frac{\partial f}{\partial z_{\ell}} \frac{\partial f}{\partial \bar{z}_{\ell}} \right) \dd z_1 \wedge \dd \bar{z}_1 \wedge \ldots \wedge  \dd z_n \wedge \dd \bar{z}_n  \]
and that
\[ \frac{\omega^n}{n!} = \frac{i^n}{2^n} \dd z_1 \wedge \dd \bar{z}_1 \wedge \ldots \wedge  \dd z_n \wedge \dd \bar{z}_n. \]
Then the result follows from the equality
\[ \sum_{\ell = 1}^n \frac{\partial f}{\partial z_{\ell}} \frac{\partial f}{\partial \bar{z}_{\ell}} = \frac{1}{2}  |\partial f|_{\omega}^2 = \frac{1}{4} |\dd f|_{\omega}^2 \]
see Equation \eqref{eq:norme} and Lemma \ref{lm:norme_df}.
\end{proof}

\subsection{Estimates outside the zero set}\label{subsec:outsidethezero}
 
In this section, we prove Theorem \ref{thm:error_term_in_k_2}. We keep the notation introduced at the beginning of Section \ref{sec:planck}.

\begin{lemma}\label{lemma loin des zeros} Let $\varphi$ be a smooth $(n-1,n-1)$-form and let $R>0$.
Then, for any $x\notin  f^{-1}(0)$, we have $$\langle\partial\bar{\partial}\log (f^2+k^{-1}b_1),\varphi_{x,R,k}\rangle=\int_{B(x,\frac{R}{\sqrt{k}})}\partial\bar{\partial}\log f^2\wedge\varphi+O(k^{-n-1})$$
 as $k \rightarrow + \infty$. 
\end{lemma}
\begin{proof}
The result follows directly from the uniform estimate 
\[\partial\bar{\partial}\log (f^2+k^{-1}b_1)=\partial\bar{\partial}\log f^2+O(k^{-1})\]
on $B(x,\frac{R}{\sqrt{k}})$ and from the volume estimate  $\mathrm{Vol}(\mathrm{Supp}\big(\varphi_{x,R,k})\big)=O(k^{-n})$.
\end{proof}

\begin{lemma}\label{lemma loin des zeros 2} Let $\varphi$ be a smooth $(n-1,n-1)$-form and let $R>0$.
Then, for any $x\notin  f^{-1}(0)$, we have 
\[ i\int_{B(x,\frac{R}{\sqrt{k}})}\partial\bar{\partial}\log f^2\wedge\varphi  = R^{2n} k^{-n} \mathrm{Vol}\big(B_{\R^{2n}}(0,1)\big) L_{\varphi}(x) + O(k^{-n-\frac{1}{2}})\]
as $k \rightarrow + \infty$, where $L_{\varphi}$ is the function defined by the equality $i\partial\bar{\partial}\log f^2\wedge\varphi = L_{\varphi} \frac{\omega^{n}}{n!}$.  
\end{lemma}

\begin{proof}  For any $z \in B(x,\frac{R}{\sqrt{k}})$,  we have
\[ L_{\varphi}(z) = L_{\varphi}(x) + O(|z-x|) = L_{\varphi}(x) + O(k^{-\frac{1}{2}}) \]
so that
\[ \begin{split} i \int_{B(x,\frac{R}{\sqrt{k}})}\partial\bar{\partial}\log f^2\wedge\varphi & = \left( L_{\varphi}(x) + O(k^{-\frac{1}{2}}) \right) \int_{B(x,\frac{R}{\sqrt{k}})}  \frac{\omega^{n}}{n!} \\
& =  \mathrm{Vol}\left(B\left(x,\frac{R}{\sqrt{k}}\right)\right) L_{\varphi}(x)+O(k^{-n-\frac{1}{2}}).  \end{split} \]
Now, we use the same argument as in the proof of Lemma \ref{lemma sur les zeros 2} to prove that 
\[ \mathrm{Vol}\left(B\left(x,\frac{R}{\sqrt{k}}\right)\right)= R^{2n} k^{-n} \mathrm{Vol}\big(B_{\R^{2n}}(0,1)\big) + O(k^{-n-\frac{1}{2}}).\]
Hence the result.
\end{proof}

\begin{proof}[Proof of Theorem \ref{thm:error_term_in_k_2}]
The proof follows the lines of the proof of Theorem \ref{thm:error_term_in_k}, using Lemmas \ref{lemma loin des zeros} and \ref{lemma loin des zeros 2} instead of Lemmas \ref{lemma sur les zeros} and \ref{lemma sur les zeros 2}.
\end{proof}

\section{Distribution of zeros of random sections and numerical simultations}\label{Sec:Random}

\subsection{Random sections of line bundles}

In this section we recall the setting of random algebraic geometry introduced in \cite{shiffman-zelditch}. We follow the notation of Section \ref{background}. In particular let $(L,h)$ be a holomorphic line bundle with positive curvature $\omega = i c_1(L,h)$ over a K\"ahler manifold $(M,\omega)$ and let $f$ be a smooth function on $M$.  
The $L^2$ Hermitian product constructed in Section \ref{framework} induces a Gaussian measure $\mu_k$ on $H^0(M,L^k)$ given by $\dd\mu_k(s)=\frac{1}{\pi^{N_k}}e^{-\norm{s}^2_{L^2}}\dd s$. Here $\dd s$ is the Lebesgue measure on $\big(H^0(M,L^k),\langle\cdot,\cdot \rangle_{L^2}\big)$ and $N_k=\dim H^0(M,L^k)$.
This Gaussian measure allows us to study the distribution-valued random variable 
\[ s\in H^0(M,L^k)\mapsto Z_{T_k s}\in\mathcal{D}^{1,1}(M). \]
 The expected value of this random variable is then defined by 
\begin{equation}\label{expecteddefinition}
\E\big[\langle Z_{T_k s},\varphi\rangle\big]=\int_{s\in H^0(M,L^k)}\left(\int_{T_k s=0}\varphi\right)\dd\mu_k(s).
\end{equation}
for any smooth $(n-1,n-1)-$form $\varphi$. In the case of $f=1$, Shiffman and Zelditch proved that
\begin{equation}\label{expectedfor1}
\frac{1}{k}\E\big[\langle Z_s,\varphi\rangle\big] = \frac{1}{2\pi} \int_M \omega \wedge \varphi +O(k^{-1}).
\end{equation}
We will obtain a similar result for $\E\big[Z_{T_k s}\big]$ by combining Theorem \ref{convergenceofcurrents} and the following standard result.

\begin{lemma}
\label{expectedforanyk}
Let $T_k$ be a Berezin-Toeplitz operator with principal symbol $f$. For any $k \in \N$, we have the following equality of currents:
\[ \E\big[Z_{T_k s}\big] = \frac{1}{2\pi} \Phi_{T_k}^*\omega_{FS}. \]
\end{lemma}

\begin{proof}
The proof follows the lines of \cite[Lemma 3.1]{shiffman-zelditch}, but we give it for the sake of completeness.
Let us fix an orthonormal basis $e_1, \dots, e_{N_k}$ of $H^0(M,L^k)$ and a local non-vanishing holomorphic section $e_L$ of $L$ defined on some open set $U\subset M$. On this open set $U$ we then have the equality $T_k e_i = f_ie_L^k$, for some holomorphic function $f_i$ defined on $U$. Thus, locally on $U$, the Kodaira map $\Phi_{T_k}$ can be read as $x\in U\mapsto (f_1(x),\dots,f_{N_k}(x))$ and the pull-back of the Fubini-Study form on $U$ equals 
\begin{equation}\label{fubini study equality}
\Phi_{T_k}^*\omega_{FS}\mid_{U}= i\partial\bar{\partial}\log \sum_{i=1}^{N_k}\abs{f_i}^2.
\end{equation} 

On the other hand, by the Poincar\'e-Lelong formula \cite[p. 388]{GH}, we have that the current defined by integration along the zero locus of a section $s=\sum_{i=1}^{N_k} a_i T_k s$ is (locally on $U$) equal to $\frac{i}{\pi}\partial\bar{\partial}\log \abs{\sum_{i=1}^{N_k}a_if_i}$.
We then have to prove that, for any smooth test $(n-1,n-1)$--form $\varphi$ with compact support in $U$, we have the following equality:
\begin{equation}\label{equality to prove}
\frac{i}{\pi}\int_{a\in\C^{N_k}}\int_{M}\partial\bar{\partial}\log \abs{\sum_{i=1}^{N_k}a_if_i}\wedge\varphi\mathrm{d}\mu_k = \frac{i}{2\pi} \int_M\partial\bar{\partial}\log \bigg(\sum_{i=1}^{N_k}\abs{f_i}^2\bigg)\wedge\varphi\mathrm{d}\mu_k.
\end{equation}
In order to prove this equality, let us denote by $\abs{f}_2$ the quantity $\bigg(\sum_{i=1}^{N_k}\abs{f_i}^2\bigg)^{\frac{1}{2}}$, so that $\abs{\sum_{i=1}^{N_k}a_if_i}$ equals $\abs{f}_2\abs{\sum_{i=1}^{N_k}a_iu_i}$, with $\sum_{i=1}^{N_k}\abs{u_i}^2=1$. The left-hand side of \eqref{equality to prove} is then equal to 
\begin{equation}\label{linearity}
\frac{i}{\pi}\int_{a\in\C^{N_k}}\int_{M}\partial\bar{\partial}\log \abs{\sum_{i=1}^{N_k}a_iu_i}\wedge\varphi\mathrm{d}\mu_k+\frac{i}{\pi}\int_{a\in\C^{N_k}}\int_M\partial\bar{\partial}\log \bigg(\sum_{i=1}^{N_k}\abs{f_i}^2\bigg)^{\frac{1}{2}}\wedge\varphi\mathrm{d}\mu_k.
\end{equation}
The function inside the integral in the second term of the sum \eqref{linearity} does not depend on $a\in\C^{N_k}$, so that
$$\frac{i}{\pi}\int_{a\in\C^{N_k}}\int_M\partial\bar{\partial}\log \bigg(\sum_{i=1}^{N_k}\abs{f_i}^2\bigg)^{\frac{1}{2}}\wedge\varphi \ \mathrm{d}\mu_k = \frac{i}{\pi}\int_M\partial\bar{\partial}\log \bigg(\sum_{i=1}^{N_k}\abs{f_i}^2\bigg)^{\frac{1}{2}}\wedge\varphi$$
which is the right-hand side of \eqref{equality to prove}. In order to prove the equality \eqref{equality to prove}, we then have to prove that the first term of the sum \eqref{linearity} is zero. In order to prove this, we use polar coordinates $a=r\theta$, for $r\in\R_+$ and $\theta=(\theta_1,\dots,\theta_{N_k})\in S^{N_k-1}$ and obtain 

\[ \begin{split} \frac{i}{\pi}\int_{a\in\C^{N_k}}\int_{M}\partial\bar{\partial}\log \abs{\sum_{i=1}^{N_k}a_iu_i}\wedge\varphi \ \mathrm{d}\mu_k & = \frac{i}{\pi}\int_{\theta\in S^{N_k-1}}\int_{M}\partial\bar{\partial}\log \abs{\sum_{i=1}^{N_k}\theta_iu_i}\wedge\varphi \ \mathrm{d}\mu_k\mathrm{d}\theta \\
& = \frac{i}{\pi}\int_{M}\partial\bar{\partial}\bigg(\int_{\theta\in S^{N_k-1}}\log \abs{\sum_{i=1}^{N_k}\theta_iu_i}\mathrm{d}\theta\bigg)\wedge\varphi \ \mathrm{d}\mu_k \\
& = 0 \end{split} \]
since the quantity $\int_{\theta\in S^{N_k-1}}\log \abs{\sum_{i=1}^{N_k}\theta_iu_i}\mathrm{d}\theta$ does not depend on $u$ for $\abs{u}=1$. Hence the result.
\end{proof}
As said in Section \ref{intro:random}, Theorem \ref{thm expectedvalue} and Theorem \ref{thm random error} follow from Lemma \ref{expectedforanyk} and from Theorems \ref{convergenceofcurrents}, \ref{thm:error_term}, \ref{thm:error_term_in_k} and \ref{thm:error_term_in_k_2}.

\subsection{Numerics}
\label{sec:numerics}

We conclude by illustrating Theorem \ref{thm random error} numerically. In order to do so, we investigate examples on the Riemann sphere; let us briefly recall the constructions in this context. For more details, see for instance \cite[Example 5.2.4, Example 7.2.5]{LFbook} and the references therein.

\paragraph{Notation.} We endow $(M, \omega) = (\C\Pp^1, \omega_{\text{FS}})$ with the line bundle $L = \mathcal{O}(1)$, equipped with the Hermitian metric $h$ which is dual to the metric on $\mathcal{O}(-1)$ coming from the standard Hermitian metric on $\C^2$. The curvature of $(\mathcal{O}(1),h)$ equals $-i \omega_{\text{FS}}$ with $\omega_{\text{FS}}$ the Fubini-Study form, normalized so that $\mathrm{Vol}(\C\Pp^1, \omega_{\text{FS}}) = 2\pi$. It is standard that for every $k \in \N$, there is a canonical isomorphism
\[ H^0(\C\Pp^1, \mathcal{O}(k)) \simeq \C_k^{\text{hom}}[z_0,z_1] \]
between the space of holomorphic sections of $\mathcal{O}(k) \to \C\Pp^1$ and the space of homogeneous polynomials of degree $k$ in two complex variables. An orthonormal basis for the $L^2$-Hermitian product obtained from this isomorphism is
\[ e_{\ell,k} = \sqrt{\frac{(k+1) \binom{k}{\ell}}{2\pi}} \ z_0^{\ell} z_1^{k-\ell}, \quad 0 \leq \ell \leq k. \]
So a random holomorphic section of $\mathcal{O}(k)$ will be of the form
\begin{equation} s_k = \sum_{\ell=0}^k \alpha_{\ell,k} e_{\ell,k}, \quad \alpha_{\ell,k} \sim \mathcal{N}_{\C}(0,1) \ \text{i.i.d.} \label{eq:rand_sect}\end{equation}
By considering the affine chart $\{ [z_0:z_1],  \ z_1 \neq 0 \}$ of $\C\Pp^1$ and the corresponding trivialization of $\mathcal{O}(1)$, we will work in the space $\C_k[z]$ of polynomials of degree at most $k$ in one complex variable, and our Berezin-Toeplitz operators will be differential operators with respect to $z$. Moreover, by symplectically identifying $(\C\Pp^1,\omega_{\text{FS}})$ with $(S^2,-\frac{1}{2} \omega_{S^2})$ where $\omega_{S^2}$ is the usual symplectic form given by
\[ (\omega_{S^2})_u(v,w) = \langle u, v \wedge w \rangle_{\R^3}, \quad u \in S^2, \ v, w \in T_u S^2,  \]
we work with symbols in $\mathscr{C}^{\infty}(S^2,\R)$. 

\paragraph{Sample mean.} In our simulations, we will consider $N$ independent random holomorphic sections $s_k^{(1)}, \ldots, s_k^{(N)} \in H^0(\C\Pp^1,\mathcal{O}(k))$ and compute the difference between the sample mean of the number of zeros of $T_k s_k$ contained in the geodesic ball $B(x,\frac{R}{\sqrt{k}})$ and $\frac{k}{2\pi} \mathrm{Vol}(B(x,\frac{R}{\sqrt{k}})$, around a point $x \in S^2$:
\begin{equation} \mathcal{E}(x,R,k,N) = \frac{1}{N} \sum_{m=1}^N \# \left( Z_{T_k s_k^{(m)}} \cap B(x,\frac{R}{\sqrt{k}}) \right) - k \left( 1 - \frac{1}{1 + \tan^2(\frac{R}{\sqrt{k}})} \right).  \label{eq:empirical} \end{equation}
Here we have used that
\[ \mathrm{Vol}\left(B\left(x,\frac{R}{\sqrt{k}}\right) \right) = 2\pi \left( 1 - \frac{1}{1 + \tan^2(\frac{R}{\sqrt{k}})} \right).  \]
For a fixed value of $k$, the random variable $ \# \left( Z_{T_k s_k^{(m)}} \cap B(x,\frac{R}{\sqrt{k}}) \right)$ is bounded, so by the law of large numbers $\mathcal{E}(x,R,k,N)$ converges almost surely towards 
\[ \mathbb{E}\left[\# \left( Z_{T_k s_k} \cap B(x,\frac{R}{\sqrt{k}}) \right)\right] - k \left( 1 - \frac{1}{1 + \tan^2(\frac{R}{\sqrt{k}})} \right) \]
as $N \to +\infty$. Recall that Theorem \ref{thm random error} applied to $\varphi = 1$ states that as $k \to +\infty$,
\begin{equation} \mathbb{E}\left[\# \left( Z_{T_k s_k} \cap B(x,\frac{R}{\sqrt{k}}) \right)\right] - k \left( 1 - \frac{1}{1 + \tan^2(\frac{R}{\sqrt{k}})} \right) = \frac{C_1(R)}{2\pi} +O(k^{-\frac{1}{2}}) \label{eq:num_on_zeros} \end{equation}
if $f(x) = 0$ and
\begin{equation} \mathbb{E}\left[\# \left( Z_{T_k s_k} \cap B(x,\frac{R}{\sqrt{k}}) \right)\right] - k \left( 1 - \frac{1}{1 + \tan^2(\frac{R}{\sqrt{k}})} \right) = k^{-1} \frac{R^{2} L_1(x)}{2}  + O(k^{-\frac{3}{2}}) \label{eq:num_out_zeros} \end{equation}
if $f(x) \neq 0$, where $L_1$ is such that $i \partial \bar{\partial} \log f^2 = L_1 \omega_{\text{FS}}$. So for a fixed but large $k$, $\mathcal{E}(x,R,k,N)$ should be close, for large $N$, either to $\frac{C_1(R)}{2\pi}$ if $f(x) = 0$ or to $ k^{-1} \frac{R^{2} L_1(x)}{2}$ if $f(x) \neq 0$. 

\paragraph{First example.} Firstly, we consider the height function $f = x_3$ on $S^2 \subset \R^3$, with $(x_1,x_2,x_3)$ the Cartesian coordinates in $\R^3$. The operator 
\[ T_k = \frac{1}{k+2} \left( 2 z \frac{\dd}{\dd z} - k \text{Id} \right) \]
acting on $\C_k[z]$ is a Berezin-Toeplitz operator with principal symbol $f$. Let $s_k$ be a random polynomial in $\C_k[z]$ as in Equation \eqref{eq:rand_sect}. Since $T_k e_{\ell,k} = \frac{\ell - 2k}{k+2} e_{\ell,k}$ for any $\ell \in \{0, \ldots, k\}$, the zeros of $T_k s_k$ are the zeros of the random polynomial
\begin{equation} T_k s_k =  \sum_{\ell=0}^k \frac{\ell - 2k}{k+2} \alpha_{\ell,k}  e_{\ell,k} \label{rand_height}\end{equation}
and can be computed numerically. So we can locate them and compute $\mathcal{E}$ as in Equation \eqref{eq:empirical}. We compare this quantity to the theoretical limits displayed in Equation \eqref{eq:num_on_zeros} and Equation \eqref{eq:num_out_zeros}. 

Since $n = 1$, the universal constant appearing in Equation \eqref{eq:num_on_zeros} is 
\begin{equation} \frac{C_1(R)}{2\pi} = 1 - \frac{1}{\sqrt{1  + 2 R^2}}.  \label{eq:C1}\end{equation}
This equality is confirmed numerically in Figure \ref{fig:R_var} by computing $\mathcal{E}(x,R,k,N)$ for some $x \in f^{-1}(0)$.

We also look at what happens outside $f^{-1}(0)$. Hence we need to compute the term $L_1$ appearing in the right-hand side of Equation \eqref{eq:num_out_zeros}. For this, note that the height function $f = x_3$ reads $f(z) = \frac{|z|^2 - 1}{|z|^2 + 1} $ in the affine holomorphic coordinate $z$, so that if $z \notin f^{-1}(0)$, then
\[ (\partial \bar{\partial} \log f^2)_z = - \frac{4(1 + |z|^4)}{(|z|^2 - 1)^2 (|z|^2 + 1)^2} \dd z \wedge \dd \bar{z}. \]
This implies, using that $\omega_{\text{FS}} = \frac{i \dd z \wedge \dd \bar{z}}{(1 + |z|^2)^2}$, that Equation \eqref{eq:num_out_zeros} becomes in this case \begin{equation} \mathbb{E}\left[\# \left( Z_{T_k s_k^{(m)}} \cap B(\pi_N^{-1}(z),\frac{R}{\sqrt{k}}) \right)\right] - k \left( 1 - \frac{1}{1 + \tan^2(\frac{R}{\sqrt{k}})} \right) = - \frac{2 k^{-1} R^{2} (1 + |z|^4)}{(|z|^2 - 1)^2} + O(k^{-\frac{3}{2}}). \label{eq:out_zeros_height} \end{equation}
This is checked in Figure \ref{fig:R_var_outside}.  

\begin{figure}[H]
  \begin{center}
    \includegraphics[width=0.9\linewidth]{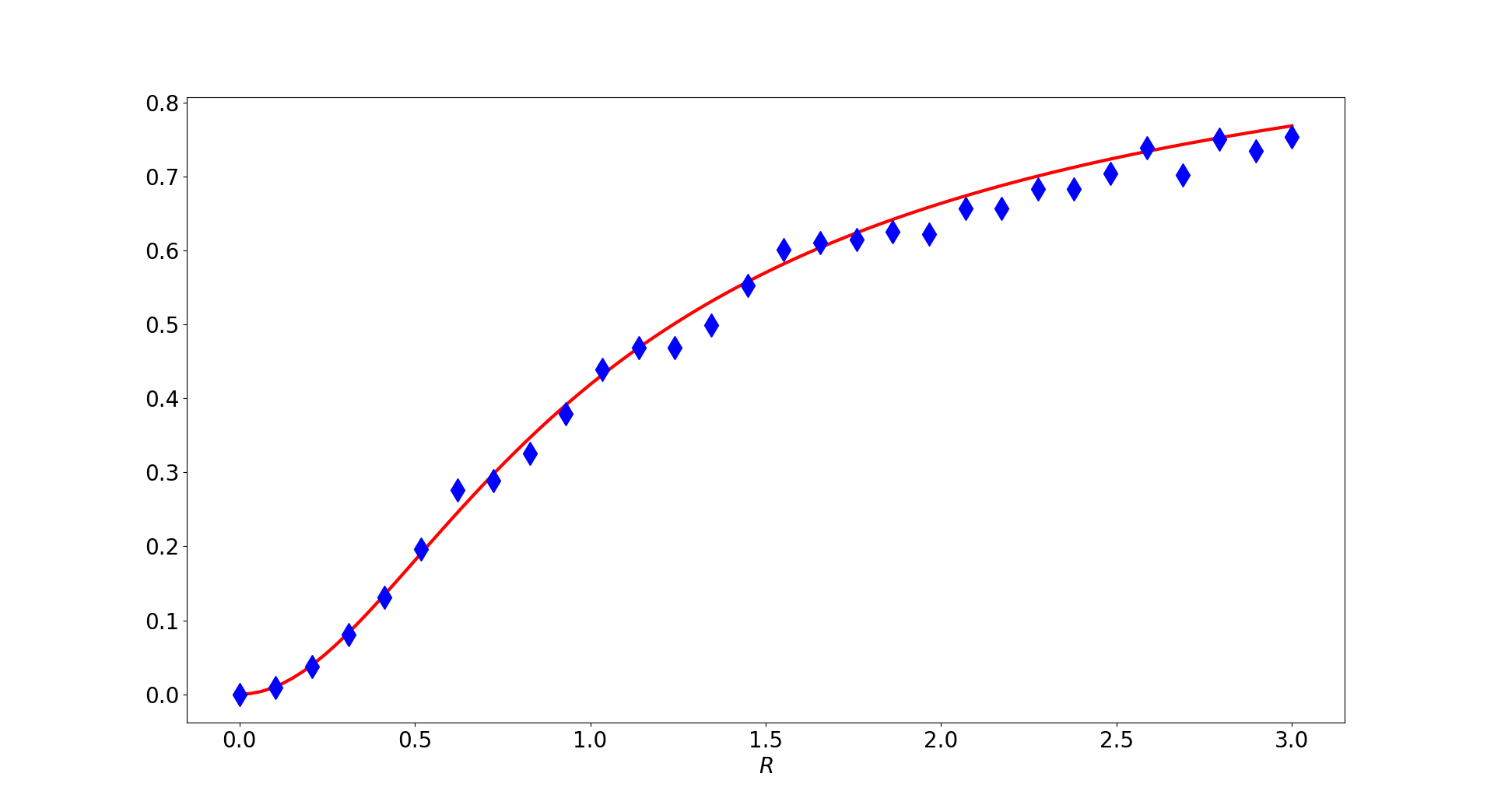} 
  \end{center}
  \caption{The blue diamonds are the numerical values of $\mathcal{E}(x,R,k,N)$ (see Equation \eqref{eq:empirical}) for $x = (1,0,0)$, $k = 400$, $N = 1000$ and various values of $R$. The solid red line is the graph of $\frac{C_1}{2\pi}$, see Equation \eqref{eq:C1}.}
  \label{fig:R_var}
\end{figure}

\begin{figure}[H]
  \begin{center}
    \includegraphics[width=0.9\linewidth]{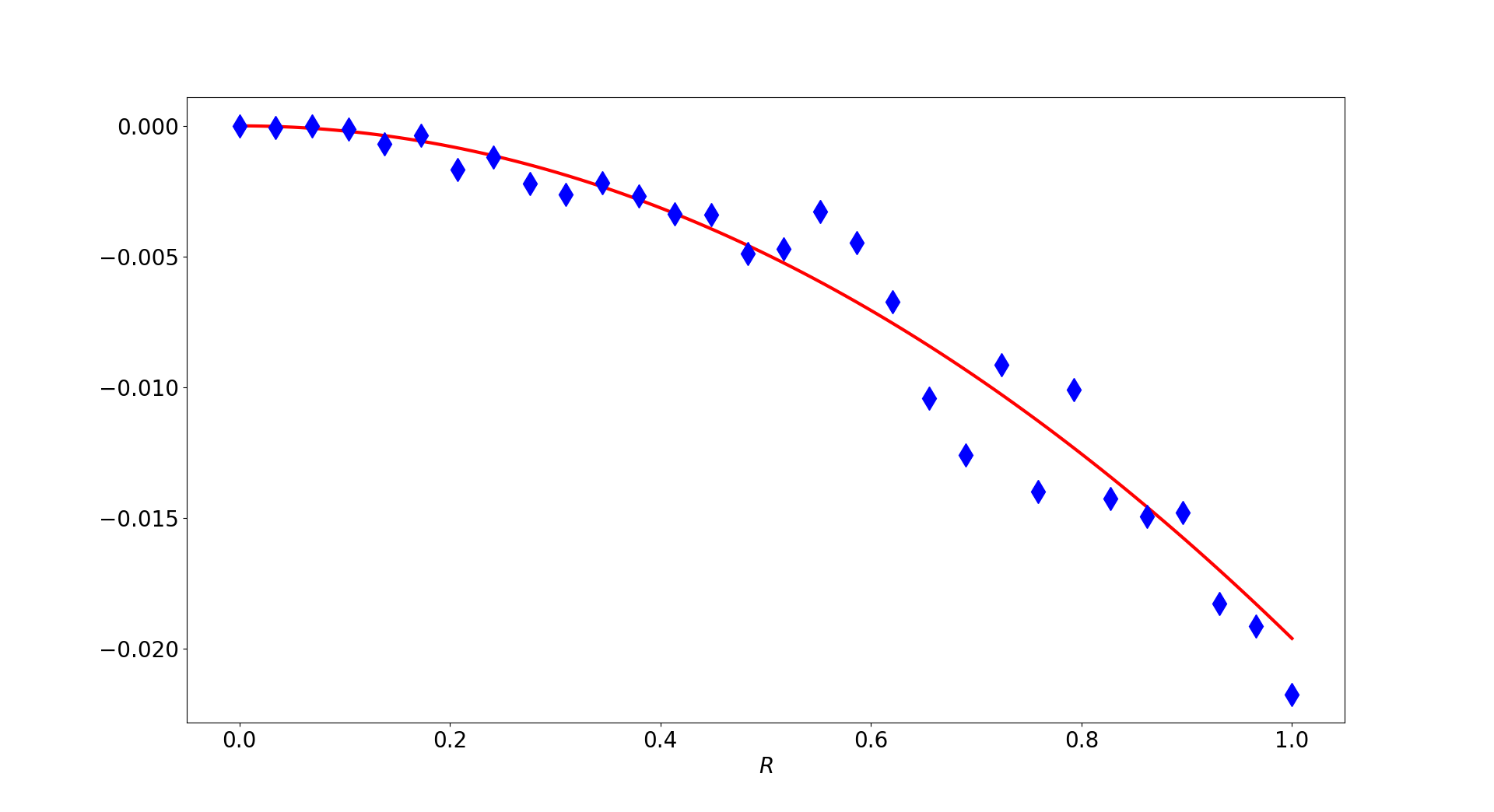} 
  \end{center}
  \caption{The blue diamonds are the numerical values of $\mathcal{E}(x,R,k,N)$ (see Equation \eqref{eq:empirical}) for $z =0$, $k = 100$, $N = 100000$ and various values of $R$. The solid red line is the graph of $R \mapsto - \frac{2 k^{-1} R^{2} (1 + |z|^4)}{(|z|^2 - 1)^2}$ for these values of $k$ and $z$, see Equation \eqref{eq:out_zeros_height}.}
  \label{fig:R_var_outside}
\end{figure}

\paragraph{Second example.} Secondly, we consider the function $f_{\lambda} = x_1 x_2 - \lambda$ on $S^2$, where $0 < \lambda < \frac{1}{2}$ (so that $f_{\lambda}$ vanishes transversally). The operator $T_k = T_k(x_1) T_k(x_2) - \lambda \text{Id}$ is a Berezin-Toeplitz operator with principal symbol $f_{\lambda}$. Using \cite[Example 5.2.4]{LFbook}, one can compute its matrix in the orthonormal basis $(e_{\ell,k})_{0 \leq \ell \leq k}$ as follows: 
\[ \forall \ell \in \{0, \ldots, k \} \qquad  T_k e_{\ell,k} = \frac{-i}{(k+2)^2} \left( \mu_{\ell,\ell-2,k} e_{\ell-2,k} - \mu_{\ell+2,\ell,k} e_{\ell+2,k} \right) - \lambda e_{\ell,k} \]
where
\[ \mu_{p,q,k} = \sqrt{p (p-1) (k - q)(k - q - 1)} \ \text{ if } p,q \in \{ 2, \ldots, k-2 \}  \]
and $\mu_{p,q,k} = 0$ otherwise. So if $s_k$ is a random holomorphic section as in Equation \eqref{eq:rand_sect}, then we compute $T_k s_k$ by applying this matrix and locate its zeros numerically. In Figure \ref{fig:inverse_xy}, we show how our results allow us to recover the set $f_{\lambda}^{-1}(0)$ by computing the quantity $\mathcal{E}(x,R,k,N)$ as in Equation \eqref{eq:empirical} for a large number of values of $x$; indeed, recall (see Equations \eqref{eq:num_on_zeros} and \eqref{eq:num_out_zeros} and the discussion after them) that for $N$ sufficiently large, this quantity is close to $\frac{C_1(R)}{2\pi} >  0$ when $f_{\lambda}(x) = 0$ and is close to a $O(k^{-1})$ otherwise.

\begin{figure}[H]
  \begin{subfigure}{.49\textwidth}
    \centering
    \includegraphics[width=\linewidth]{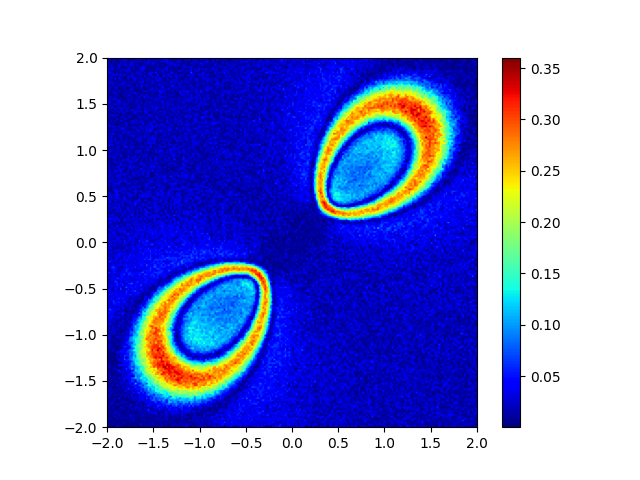}
  \end{subfigure}
  \begin{subfigure}{.49\textwidth}
    \centering
    \includegraphics[width=0.68\linewidth]{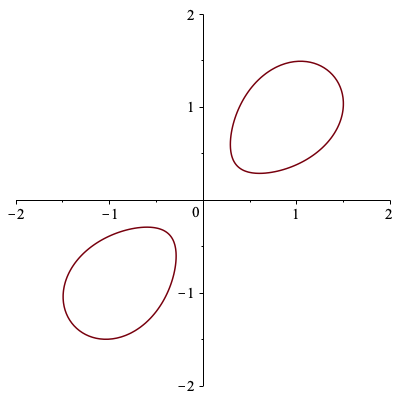}
  \end{subfigure}
  \caption{Reconstruction of the set $f_{\lambda}^{-1}(0)$ for $f_{\lambda} = x_1 x_2 - \lambda$ on $S^2$, with $\lambda = \frac{1}{3}$, after stereographic projection. On the left we display the values of $\mathcal{E}(z,R,k,N)$ (see Equation \eqref{eq:empirical}) for $R = \frac{1}{\sqrt{2}}$, $k = 100$, $N = 1000$, and $z$ taken on a $200 \times 200$ grid discretizing the square $\{ |\Re(z)|, |\Im(z)| \leq 2 \}$. On the right we show the level set $f_{\lambda}^{-1}(0)$ for $\lambda = \frac{1}{3}$.}
  \label{fig:inverse_xy}
\end{figure}

\appendix

\section{Appendix: a proof of Theorem \ref{symbol_comp}}
\label{sect:appendix}

In this appendix, we show how to derive Theorem \ref{symbol_comp} from \cite{Cha03}. Once again, we stress that Theorem \ref{symbol_comp} already exists in the literature (see for instance \cite{Cha03,ma-marinescu-toeplitz}). Our goal here is to write a proof with our notation and conventions (which differ from those of \cite{Cha03} and \cite{ma-marinescu-toeplitz}) as the explicit values of the constants appearing in the statement have been intensively used throughout the paper. 

\begin{proof}[Proof of Theorem \ref{symbol_comp}]
Since $T_k - k^{-1} T_k(f_1)$ and $S_k - k^{-1} T_k(g_1)$ are Berezin-Toeplitz operators with respective principal symbols $f$ and $g$ and vanishing subprincipal symbols, it suffices to consider the case $f_1 = 0 = g_1$. Moreover the terms of order $k^{-\ell}$, $\ell \geq 2$ in the symbols of $T_k$ and $S_k$ do not contribute to $b_1$, so we may assume that $T_k = T_k(f_0)$ and $S_k = T_k(g_0)$. We write the asymptotic expansion of the kernel of $B_k$ on the diagonal as 
\[ B_k(x,x) = \sum_{\ell \geq 0} k^{-\ell} b_{\ell}(f_0,g_0)(x) + O(k^{-\infty}). \]
We want to compute the term $b_1$ in the symbol
\[ \sigma(B_k) = \sum_{\ell \geq 0} \hbar^{\ell} b_{\ell}(f_0,g_0). \]
Recall from \cite{Cha03} that the covariant and contravariant symbols of $T_k$ are defined as
\[ \sigma_{\text{cov}}(T_k) = \sigma(T_k) \sigma(\Pi_k)^{-1}, \quad \sigma_{\text{cont}}(T_k(f)) = f. \]
The associated star-products, $\star$, $\star_{\text{cov}}$ and $\star_{\text{cont}}$ are
\[ \sigma(ST) = \sigma(S) \star \sigma(T), \quad \sigma_{\text{cov}}(ST) = \sigma_{\text{cov}}(S) \star_{\text{cov}} \sigma_{\text{cov}}(T), \quad \sigma_{\text{cont}}(ST) = \sigma_{\text{cont}}(S) \star_{\text{cont}} \sigma_{\text{cont}}(T). \]
So by definition, 
\[ \sigma(T_k(f_0) T_k(g_0)) = \sigma_{\text{cov}}(T_k(f_0) T_k(g_0)) \sigma(\Pi_k). \]
Let $\Psi$ be the isomorphism sending $\sigma_{\text{cont}}$ to $\sigma_{\text{cov}}$, so that
\[ \sigma(T_k(f_0) T_k(g_0)) = \Psi(\sigma_{\text{cont}}(T_k(f_0) T_k(g_0))) \sigma(\Pi_k) = \Psi(f_0 \star_{\text{cont}} g_0) \sigma(\Pi_k). \]
We know from \cite[Proposition 4]{Cha03} that $\Psi(f) = f + \hbar \Delta f + O(\hbar^2)$ and that
\[ f_0 \star_{\text{cont}} g_0 = f_0 g_0 - 2 \hbar \sum_{\ell, m = 1}^n G^{\ell,m} \frac{\partial f_0}{\partial z_{\ell}} \frac{\partial g_0}{\partial \bar{z}_m} + O(\hbar^2)  \]
(note the different conventions for $G_{\ell,m}$ between \cite{Cha03} and the present paper). Furthermore, it is stated in \cite[Corollary 2]{Cha03} that
\[ \sigma(\Pi_k) = 1 + \hbar \frac{r}{2} + O(\hbar^2) \]
where $r$ is the scalar curvature of $M$. Consequently,
\[ \begin{split} \sigma(T_k(f_0) T_k(g_0)) & = \left( f_0 \star_{\text{cont}} g_0 + \hbar \Delta(f_0 \star_{\text{cont}} g_0) \right) \left( 1 + \hbar \frac{r}{2} + O(\hbar^2) \right) \\
& = \left( f_0 g_0 - 2\hbar \sum_{\ell, m = 1}^n G^{\ell,m} \frac{\partial f_0}{\partial z_{\ell}} \frac{\partial g_0}{\partial \bar{z}_m} + \hbar \Delta(f_0 g_0) + O(\hbar^2) \right) \left( 1 + \hbar \frac{r}{2} + O(\hbar^2) \right) \\
& = f_0 g_0 + \hbar \left( \Delta(f_0 g_0) -2 \sum_{\ell, m = 1}^n G^{\ell,m} \frac{\partial f_0}{\partial z_{\ell}} \frac{\partial g_0}{\partial \bar{z}_m}  + \frac{r f_0 g_0}{2} \right) + O(\hbar^2). \end{split} \]
But one readily checks that 
\[ \Delta(fg) = f_0 \Delta g_0 + g_0 \Delta f_0 + 2 \sum_{\ell, m = 1}^n G^{\ell,m} \frac{\partial f_0}{\partial z_{\ell}} \frac{\partial g_0}{\partial \bar{z}_m} + 2 \sum_{\ell, m = 1}^n G^{\ell,m} \frac{\partial g_0}{\partial z_{\ell}} \frac{\partial f_0}{\partial \bar{z}_m}. \]
Consequently, we finally obtain that 
\[ \sigma(T_k(f_0) T_k(g_0)) = f_0 g_0 + \hbar \left( f_0 \Delta g_0 + g_0 \Delta f_0 + \frac{r f_0 g_0}{2} + 2 \sum_{\ell, m = 1}^n G^{\ell,m} \frac{\partial g_0}{\partial z_{\ell}} \frac{\partial f_0}{\partial \bar{z}_m}  \right) + O(\hbar^2).  \]
But, in view of Equation \eqref{eq:norme} and since $f_0$ and $g_0$ are real-valued,
\[ 2 \sum_{\ell, m = 1}^n G^{\ell,m} \frac{\partial g_0}{\partial z_{\ell}} \frac{\partial f_0}{\partial \bar{z}_m} = G(\partial g_0, \partial f_0). \]
\end{proof}

\bibliographystyle{plain}
\bibliography{toeplitz}
\end{document}